\documentclass{article}
\usepackage{amsmath,amssymb,amsthm, tikz-cd, bbm, url}
\usepackage[shortlabels]{enumitem}
\usepackage{graphicx}
\usepackage[margin=1in]{geometry}

\theoremstyle{theorem}
\newtheorem{theorem}{Theorem}[section]
\newtheorem*{theorem*}{Theorem}
\newtheorem{corollary}[theorem]{Corollary}
\newtheorem*{corollary*}{Corollary}
\newtheorem{lemma}[theorem]{Lemma}
\newtheorem{proposition}[theorem]{Proposition}
\newtheorem*{thmdef*}{Theorem/Definition}

\theoremstyle{definition}
\newtheorem{remark}[theorem]{Remark}

\newtheorem{example}[theorem]{Example}
\newtheorem{definition}[theorem]{Definition}
\newtheorem*{definition*}{Definition}

\newtheorem{question}[theorem]{Question}
\newcommand{\C}{\mathbb{C}}
\newcommand{\Z}{\mathbb{Z}}

\newcommand{\cC}{\mathcal{C}}

\newcommand{\cA}{\mathcal{A}}

\newcommand{\R}{\mathbb{R}}

\newcommand{\lp}{\left(}
\newcommand{\rp}{\right)}
\newcommand{\sq}{\subseteq}
\newcommand{\ds}{\dots}
\newcommand{\cds}{\cdots}
\newcommand{\arrow}{arrow}
\newcommand{\F}{\mathbb{F}}
\newcommand{\Q}{\mathbb{Q}}
\newcommand{\N}{\mathbb{N}}
\newcommand{\A}{\mathbb{A}}

\newcommand{\fa}{\mathfrak{a}}

\renewcommand{\t}[1]{\widetilde{#1}}

\newcommand{\ul}[1]{\underline{#1}}

\DeclareMathOperator{\Hom}{Hom}
\DeclareMathOperator{\End}{End}

\DeclareMathOperator{\Spec}{Spec}

\DeclareMathOperator{\im}{im}

\DeclareMathOperator{\Mod}{Mod}

\DeclareMathOperator{\Fun}{Fun}

\newcommand{\Mustata}{Musta\c{t}\u{a}}

\setlist[enumerate]{itemsep=2pt, topsep=2pt, itemindent=10pt, label=(\roman*)}

\begin{document}
 	\title{Bernstein-Sato theory for arbitrary ideals in positive characteristic}
 	\author{Eamon Quinlan-Gallego \footnote{Partially supported by NSF grant DMS-1801697 and by the Ito Foundation for International Education Exchange}}
 	
 	\maketitle
 	
\begin{abstract}
\Mustata \ defined Bernstein-Sato polynomials in prime characteristic for principal ideals and proved that the roots of these polynomials are related to the $F$-jumping numbers of the ideal. This approach was later refined by Bitoun. Here we generalize these techniques to develop analogous notions for the case of arbitrary ideals and prove that these have similar connections to $F$-jumping numbers.
\end{abstract}
\section{Introduction} \label{scn-intro}

Let $R = \C[x_1, \ds, x_n]$ be a polynomial ring over $\C$. We denote by $D_R$ the ring of $\C$-linear differential operators on $R$, i.e. the ring generated by $R$ and its derivations inside of $\End_\C(R)$. Let $f \in R$ be a nonzero polynomial. Bernstein \cite{Ber} and Sato \cite{SatoM} independently, and in different contexts, discovered the following fact: there is a nonzero polynomial $b(s) \in \C[s]$ and a differential operator $P(s) \in D_R[s]$ satisfying the following functional equation:
$$P(s) \cdot f^{s+1} = b(s) f^s.$$
The monic polynomial $b_f(s)$ of least degree for which there is some $P(s) \in D_R[s]$ satisfying the above equation is called the Bernstein-Sato polynomial for $f$. 

In the case where $f$ defines an isolated singularity it was proven by Malgrange \cite{Mal75} that the roots of $b_f(s)$ are negative and rational. This was later extended by Kashiwara \cite{Kas76} to the case of arbitrary $f$ by using resolution of singularities. In particular, we have that $b_f(s) \in \Q[s]$.

Since its inception the Bernstein-Sato polynomial has seen a wide variety of applications. In \cite{Mal74} Malgrange exhibited a relation between the roots of $b_f(s)$ and the eigenvalues of the monodromy action on the cohomology of the Milnor fibre of $f$. Kashiwara \cite{Kas83} and Malgrange \cite{Mal83} also used the existence of Bernstein-Sato polynomials to define $V$-filtrations with the purpose of defining nearby and vanishing cycles at the level of $D$-modules. Coming full circle, Budur, \Mustata \ and Saito then used this theory of $V$-filtrations to define the Bernstein-Sato polynomial $b_\fa(s)$ of an arbitrary ideal $\fa \sq R$.

A key application of the theory of Bernstein-Sato polynomials, critical in our motivation, comes from the relationship between its roots and the jumping numbers for multiplier ideals. Since this relationship holds for the Bernstein-Sato polynomials of \cite{BMSa} (i.e. those of general ideals $\fa \sq R$) let us explain it in this more general case. 

Let $\fa \sq R$ be an ideal and suppose that $\lambda > 0$ is a real number. Using a log-resolution of $\fa$ one can define the multiplier ideal $J(\fa^\lambda)$ of $\fa$ with exponent $\lambda$ (see \cite{lazarsfeld2004positivity} for details). The $J(\fa^\lambda)$ are ideals of $R$ satisfying the following two properties:
\begin{enumerate}
	\item If $\lambda < \mu$ then $J(\fa^\lambda) \supseteq J(\fa^\mu)$. 
	\item For all $\lambda > 0$ there exists some $\epsilon > 0$ such that $J(\fa^\lambda) = J(\fa^\mu)$ for all $\mu \in [\lambda, \lambda + \epsilon].$
\end{enumerate}
An real number $\lambda$ for which one has $J(\fa^\lambda) \neq J(\fa^\mu)$ for all $\mu < \lambda$ is called a jumping number for $\fa$. The set of jumping numbers is discrete and rational, and the smallest jumping number $\alpha_\fa$ (the so-called log-canonical threshold of $\fa$) is an important invariant in singularity theory. The connection to the theory of Bernstein-Sato polynomials comes from the following fact: $\alpha$ is the smallest root of $b_\fa(-s)$, and every jumping number in the interval $[\alpha, \alpha + 1)$ is a root of $b_\fa(-s)$ \cite{Kollar97}, \cite{ELSV04}, \cite[Thm. 2]{BMSa}.

A characteristic $p > 0$ analogue of the multiplier ideal $J(\fa^\lambda)$ is given by the test ideal $\tau(\fa^\lambda)$ (c.f. Definition 2.4). The notion of test ideal originally comes from the theory of tight closure of Hochster and Huneke \cite{HH90}, and was later generalized by Hara and Yoshida in \cite{HY03}. The test ideals $\tau(\fa^\lambda)$ satisfy properties (i) and (ii) and thus one can analogously define jumping numbers for the test ideal (so-called $F$-jumping numbers of $\fa$), which are therefore characteristic-$p$ analogues of jumping numbers for the multiplier ideal. The set of $F$-jumping numbers is known to be discrete and rational (see \cite{BMSm2008} for the finite-type regular case, and \cite{SchTucTI} for greater generality).

As mentioned, in characteristic zero the Bernstein-Sato polynomial $b_\fa(s)$ carries information about the jumping numbers of $\fa$. The existence of characteristic-$p$ analogues of jumping numbers suggests that there is a Bernstein-Sato theory in positive characteristic. This hope is emboldened by the fact that the Bernstein-Sato polynomial $b_\fa(s)$ of $\fa$ (in characteristic zero) has been related to other characteristic-$p$ invariants of mod-$p$ reductions of $\fa$ \cite{MTW}.

In \cite{Mustata2009}, \Mustata \ began this line of research and developed a Bernstein-Sato theory in prime characteristic for principal ideals $\fa = (f)$. Since then this technique has been refined by Bitoun \cite{Bitoun2018} and has been extended to the setting of unit $F$-modules \cite{Stad12} and $F$-regular Cartier modules \cite{BliStab16}. In this work we provide the first instance of Bernstein-Sato theory in positive characteristic where $\fa$ is an arbitrary ideal. We do this by generalizing the work of \Mustata \ and some of the work of Bitoun and as such we still work in the ring setting; that is, we do not consider $F$-modules or Cartier modules although an extension of our work to these settings might be possible.

Let us summarize the work of \Mustata. Suppose that $R$ be a regular $F$-finite ring and that $\fa = (f)$ is a principal ideal of $R$. Over $\C$ the Bernstein-Sato polynomial is defined in \cite{BMSa} as the minimal polynomial of an operator $s_1$ acting on a certain module $N_f$, and \Mustata \ observed that the module can be defined in prime characteristic but that, instead of considering a single operator $s_1$, one should consider the action of an infinite family $\{s_{p^i} : i = 0,1, \ds \}$. Moreover, in characteristic $p> 0$ the module $N_f$ can be expressed as a direct limit
$$N_f = \lim_{\to e} N^e_{f},$$
where for all $e$ the module $N^e_{f}$ carries an action of the first $e$ operators $s_{p^0}, \ds , s_{p^{e-1}}$ and the transition map $N^e_f \to N^{e+1}_f$ is linear with respect to these. 

The operators $s_{p^i}$ satisfy two crucial properties: they are pairwise commuting and satisfy $s_{p^i}^p = s_{p^i}$. Observe that because we are in characteristic $p >0$ the latter is equivalent to $\prod_{j = 0}^{p-1} (s_{p^i} - j) = 0$. It then follows that if $e > 0$ is fixed then every module over $\F_p[s_{p^0}, \ds, s_{p^{e-1}}]$ splits as a direct sum of multi-eigenspaces. In particular, for all $e > 0$ we have decompositions
$$N^e_f = \bigoplus_{\alpha \in \F_p^e} (N^e_f)_\alpha,$$
where, given $\alpha = (\alpha_0, \ds, \alpha_{e-1}) \in \F_p^e$, the subspace $(N^e_f)_\alpha$ consists of those elements $u \in N^e_f$ for which $s_{p^i} \cdot u = \alpha_i \cdot u$ for all $i = 0, 1, \ds, e-1$. In \cite{Mustata2009} \Mustata \ proves 
that the modules $N^e_f$ carry the information about the $F$-jumping numbers of $f$ (c.f. Definition \ref{def-Fjn}) in the following way.

\begin{theorem*}[\cite{Mustata2009}] \label{mustata}
	Let $e > 0$ and $\alpha = (\alpha_0, \ds, \alpha_{e-1}) \in \F_p^e$. Then the following are equivalent.
	\begin{enumerate}[(1)]
		\item The module $(N^e_f)_\alpha$ is nonzero.
		\item There is an F-jumping number of $f$ in the interval
		$$\bigg(\frac{\alpha_0 + p \alpha_1 + \cds + p^{e-1} \alpha_{e-1}}{p^e}, \frac{\alpha_0 + p \alpha_1 + \cds + p^{e-1} \alpha_{e-1} + 1}{p^e} \bigg]$$
	\end{enumerate}
\end{theorem*}

In (2) we have implicitly identified each $\alpha_i \in \F_p$ with its unique representative in $\{0, \ds, p-1\}$. Also note that condition (2) is equivalent to the existence of an F-jumping number $\lambda \in (0,1]$ of $f$ whose base-$p$ expansion starts with $\lambda = \alpha_{e-1} p^{-1} + \alpha_{e-2} p^{-2} + \cds + \alpha_0 p^{-e}$ -- where we take the base-$p$ expansion that is eventually nonzero. \Mustata \ defines a collection of polynomials associated to $f$. He calls these Bernstein-Sato polynomials, but we have found that the na\"ive extension of this definition to the monomial case does not give the desired result and thus we prefer the name ``approximating polynomials". 
\begin{definition*} [\cite{Mustata2009}]
	Let $e > 0$. The $e$-th approximating polynomial of $f$ (called Bernstein-Sato polynomials in \cite{Mustata2009}, and denoted $b^e_f(s)$) is
	$$a^e_f(s) := \prod_{\{\alpha \in \F_p^e : (N^e_\fa)_\alpha \neq 0\}} \bigg(s - \frac{\alpha_0 + p \alpha_1 + \cds + p^{e-1} \alpha_{e-1}}{p^e}\bigg).$$
\end{definition*}
\Mustata's theorem then gives the following.
\begin{corollary*}[\cite{Mustata2009}]
	For every $e >0$ the roots of $a^e_f(s)$ are given by the rational numbers $\frac{\lceil p^e \lambda \rceil -1 }{p^e}$ as $\lambda$ ranges through all F-jumping numbers of $f$ in $(0,1]$.
\end{corollary*}

In Section \ref{scn-multi-eigen-dec} we follow \Mustata's approach and define approximating polynomials $a^e_\fa(s)$ for the case of an arbitrary ideal $\fa$ in positive characteristic. We then show that these polynomials still retain information about the F-jumping numbers of $\fa$. First of all, the roots of $a^e_\fa(s)$ are related to the $\nu$-invariants $\nu^J_\fa(p^e)$ of \cite{MTW} (c.f. Definition \ref{def-nu-invt}) as follows.
\begin{theorem*}[\ref{thm-roots-of-be}]
	Let $e > 0$ be an integer. Then all roots of $a^e_\fa(s)$ are simple and lie in $[0,1) \cap \Z \frac{1}{p^e}$. Moreover, the roots are given by $\frac{\nu}{p^e} - \lfloor \frac{\nu}{p^e} \rfloor$ where $\nu$ ranges through all $\nu$-invariants $\nu^J_\fa(p^e)$. 
\end{theorem*} 
Secondly, if $e_0$ is large enough (more precisely, a stable exponent for $\fa$, see Definition \ref{def-stable-exp}) then the roots of $a^{e_0 + e}_\fa(s)$ give approximations to the decimal parts of the $F$-jumping numbers of $\fa$, and these approximations are accurate in the order of $1/p^e$.

\begin{theorem*}[\ref{thm-roots-approx-fjn}]
	Let $e_0$ be a stable exponent and fix integers $e > 0$ and $0 \leq k < p^e$. Then the following are equivalent.
	\begin{enumerate}[(1)]
		\item There is a root of $a^{e_0 + e}_\fa(s)$ in $[\frac{k}{p^e}, \frac{k + 1}{p^e})$.
		\item There is an $F$-jumping number of $\fa$ in $(\frac{k}{p^e}, \frac{k + 1}{p^e}] + \N$.
	\end{enumerate}	
\end{theorem*}
Our second goal is to generalize some the results of Bitoun from \cite{Bitoun2018}, which we describe briefly. First of all observe that as $e$ gets bigger the approximating polynomials of \Mustata \ give increasingly better approximations to the F-jumping numbers of $f$. We would expect then that if we consider the direct limit $N_f= \varinjlim_e N^e_f$ one may be able to recover the F-jumping exponents of $f$ from this module only. To state Bitoun's theorem let us introduce some notation.

Given a $p$-adic integer $\beta \in \Z_p$ we let $\beta_i$ be the unique integers with $0 \leq \beta_i < p$ such that $\beta = \beta_0 + p \beta_1 + p^2 \beta_2 + \cds. $ We then define
$$(N_f)_\beta := \{ u \in N_f : s_{p^i} \cdot u = \beta_i u \text{ for all } i\}.$$
Bitoun's result is then as follows.

\begin{theorem*}
	[\cite{Bitoun2018}] We have a decomposition $N_f = \bigoplus_{\beta \in \Z_p} (N_f)_\beta$ and, moreover,
	$$\{\beta \in \Z_p : (N_f)_\beta \neq 0 \} = - FJ(\fa) \cap (0,1] \cap \Z_{(p)},$$
	where $FJ(\fa)$ is the collection of $F$-jumping numbers of $\fa$. 
\end{theorem*}
In \cite{Bitoun2018} the Bernstein-Sato polynomial of $\fa$ is defined as an ideal in the algebra $\Fun^{cts}(\Z_p, \F_p)$ of continuous functions from $\Z_p$ to $\F_p$. A notion of root for such an ideal is also developed in such a way that the $p$-adic numbers $\beta$ for which $(N_f)_\beta \neq 0$ are the roots of the Bernstein-Sato polynomial. 
In Section \ref{scn-Ap-and-modules} we develop some abstract theory with the goal of generalizing Bitoun's result to arbitrary $\fa$, which we achieve in Section \ref{scn-A-module-N}. For simplicity we have chosen to sidestep Bitoun's definition of Bernstein-Sato polynomial and, instead, we define its roots directly, which we call Bernstein-Sato roots.
\begin{thmdef*} \
	\begin{enumerate}[(a)]
		\item (Prop. \ref{N-is-discrete})We have a decomposition $N_\fa = \bigoplus_{\beta \in \Z_p} (N_\fa)_\beta$ and, moreover, the set $BS(\fa) := \{\beta \in \Z_p : (N_\fa)_\beta \neq 0\}$ is finite. We call an element of $BS(\fa)$ a Bernstein-Sato root.
		
		\item (Thm. \ref{BSroots-are-rational})The Bernstein-Sato roots of $\fa$ are rational and negative.
		
		\item (Thm. \ref{thm-bsroots-and-fjn}) We have
		$$BS(\fa) + \Z = -FJ(\fa)\cap \Z_{(p)} + \Z$$
		where $FJ(\fa)$ is the collection of $F$-jumping numbers of $\fa$.
	\end{enumerate}
\end{thmdef*}
The definition of Bernstein-Sato roots is better behaved than one might expect at first glance. Firstly, in a certain sense the definition given above is compatible with the one in characteristic-zero (see Subsection \ref{subscn-first-rem}). Moreover, in \cite{QG19b} we show, using results from \cite{BMSa06b}, that if $\fa \sq \Z[x_1, \ds, x_n]$ is a monomial ideal then the set of roots of $b_{\fa_\C}(s)$ (the Bernstein-Sato polynomial of the expansion of $\fa$ to $\C[x_1, \ds, x_n]$) coincides with the set of Bernstein-Sato roots (in the above sense) of $\fa_p$, the image of $\fa$ in $\F_p[x_1, \ds, x_n]$, for $p$ large enough. In Proposition \ref{prop-BSroots-newChar} we are also able to give a characterization of the Bernstein-Sato roots of $\fa$ purely in terms of the $\nu$-invariants of \Mustata, Takagi and Watanabe \cite{MTW}. 

Finally, let us remark that in \cite{Bitoun2018} Bitoun is able to give $N_f$ a unit $F$-module structure. This was later generalized by St\"abler to the setting of $F$-regular Cartier modules in \cite{StabAG}. We do not pursue the question of whether a similar statement can be made for $N_\fa$ in this paper.

Let us set up some notation. Given a $k$-algebra $R$ we denote by $D_{R|k}$, or simply by $D_R$, its ring of $k$-linear differential operators in the sense of Grothendieck \cite[\S16.8]{EGAIV}. We say that a ring $R$ of characteristic $p >0$ is $F$-finite if $R$ is finite as a module over its subring $R^p$. Whenever $R$ is an $F$-finite ring of characteristic $p>0$ we write $D^e_R := \End_{R^{p^e}}(R)$; in this setting we have that $D_{R|k} = \cup_{e = 0}^\infty D^e_R$ for every perfect field $k$ contained in $R$ \cite{Yek} \cite[Rmk. 2.5.1]{SVdB97} (if no reference to $k$ is made, we always take $k = \F_p$). 

We will use multi-index notation. That is, given a tuple $\ul{b} = (b_1, \ds, b_r)$ of integers and a tuple $ g= (g_1, \ds, g_r)$ of elements in a ring we have $g^{\ul{b}} = g_1^{b_1} \cds g_r^{b_r}$. We also use multi-index notation on binomial coefficients: given another tuple $\ul{a}$ we have ${\ul{a} \choose \ul{b}} := \prod_{j=1}^r {a_j \choose b_j}$. We denote by $\mathbbm{1}$ the tuple $\mathbbm{1} := (1,1,\ds, 1)$. Finally, we denote $|\ul{b}| = b_1 + \cds + b_r$.

\subsection*{Acknowledgments}
I would like to thank Josep \`{A}lvarez-Montaner, Alberto F. Boix, Zhan Jiang, Jack Jeffries, Mircea \Mustata, Kenta Sato, Kohsuke Shibata, Karen Smith, Axel St\"{a}bler and Shunsuke Takagi for many enjoyable discussions and great suggestions about this project. I am especially grateful to Shunsuke Takagi for suggesting this problem to me, as well as for his guidance and encouragement. 
\section{Background}

\subsection{V-filtrations} \label{V-filt}

Let $S$ be a regular $F$-finite ring of characteristic $p>0$ and $I \sq S$ be an ideal. As mentioned previously, in this setting we have $D_S = \cup_{e = 0}^\infty D^e_S$ where $D^e_S := \End_{S^{p^e}}(S)$.

We begin by observing that the $V$-filtration of Kashiwara and Malgrange on the ring of differential operators can be defined in characteristic $p$. Moreover, in this setting the ideal $I$ induces decreasing filtrations $V$ on $D_S$ and all its subrings $D^e_S$ by setting, for all $e \geq 0$ and $i \in \Z$,
$$V^i D^e_S := \{\xi \in D^e_S | \xi \cdot I^j \sq I^{j + i} \text{ for all } j \in \Z\}$$
and similarly for $D_S$. We use the convention that $I^n = S$ whenever $n \leq 0$. Observe that the filtration $V$ is indexed by $\Z$ and decreasing.
We will only be concerned with $V$-filtrations arising in the following setting: let $R$ be an $F$-finite regular ring of characteristic $p>0$, $S =R[t_1, \ds , t_r]$ be a polynomial ring over $R$ and $I = (t_1, \ds, t_r) \sq S$. We can give $S$ an $\N_0$-grading by setting $\deg t_i = 1$ for all $i$. This induces a $\Z$-grading on all $D^e_S$ (since $S$ is $F$-finite) and thus on $D_S$. Given $d \in \Z$ we denote by $S_d$ (resp. $(D^e_S)_d$,  $(D_S)_d$) the set of homogeneous elements of degree $d$ in $S$ (resp. $D_S$, $D^e_S$). 
\begin{lemma} \label{V-filt-lemma}
Let $R$, $S$ and $I$ be as above and let $i \in \Z$. 
\begin{enumerate}[(a)]
	\item We have $V^i D^e_S = (D^e_S)_{\geq i}$ for all $e \geq 0$ and $V^i D_S = (D_S)_{\geq i}$.
	\item Suppose $i \geq 0$. Then $(D^e_S)_i = S_i (D^e_S)_0 = (D^e_S)_0 S_i$.
	\item Suppose $i \geq 0$. Then $V^i D^e_S = I^i V^0 D^e_S = V^0 D^e_S I^i$. 
\end{enumerate}
\end{lemma}
\begin{proof}
	We first prove (a). Let $e>0$, and we will show $V^i D^e_S = (D^e_S)_{\geq i}$. The result for $D_S$ will follow.
	
	The inclusion $(\supseteq)$ is clear. For $(\sq)$ let $\xi \in V^i D^e_S$ and let $\xi = \sum_d \xi_d$ where $\xi_d$ is homogeneous of degree $d$. By subtracting we may assume that $d < i$ whenever $\xi_d \neq 0$. Suppose for a contradiction that $\xi \neq 0$ and let $d_0 := \max \{d : \xi_d \neq 0\}$. As $\xi_{d_0} \neq 0$ there exists some homogeneous $g \in S$ of degree $m$ such that $\xi_{d_0}(g) \neq 0$. By multiplying by $t_1^{p^{le}}$ for some sufficiently large $l$ we may assume that $m > -d_0$. 
	
	On the one hand, since $\xi \in V^i D^e_S$ we have $\sum \xi_d(g) = \xi(g) \in I^{m + i}$. On the other hand, the homogeneous component $\xi_{d_0}(g)$ of $\xi(g)$ has degree $m + d_0 < m + i$. As $\xi_{d_0}(g) \neq 0$, this is a contradiction.
	
	Let us now prove (b). Let $P = \F_p[t_1, \ds, t_r]$ with $\deg t_i = 1$, so that $S = R \otimes_{\F_p} P$. There is a natural isomorphism $D^e_R \otimes_{\F_p} D^e_P \xrightarrow{\sim} D^e_S$ which identifies $(D^e_S)_i \cong D^e_R \otimes_{\F_p} (D^e_{P})_i$ for all $i \in \Z$ and thus it suffices to prove the statement for $S = P$, i.e. $R = \F_p$.
	
	We thus assume $S = P$. Let $L := \{\alpha \in \N^r : 0 \leq \alpha_j < p^e \}$ and, given $\alpha \in L$ denote $|\alpha| := \alpha_1 + \cds + \alpha_r$. Observe that an operator in $D^e_S$ is uniquely determined by its action on monomials $t^\alpha$ with $\alpha \in L$. Given $\alpha \in L$ we define $\sigma_\alpha \in D^e_S$ by $\sigma_\alpha \cdot t^\alpha = 1$ and $\sigma_\alpha \cdot t^\beta = 0$ for all $\alpha \neq \beta \in L$. With this notation we have $D^e_S = \bigoplus_{\alpha \in L} S \sigma_\alpha$ and therefore $(D^e_S)_i = \bigoplus_{\alpha \in L} S_{|\alpha| + i} \sigma_\alpha = S_i (D^e_S)_0$, which gives the first equality.
	
	We now prove the second equality. First observe that the containment $(D^e_S)_0 S_i \sq (D^e_S)_i$ is clear. Next notice that by our previous discussion $(D^e_S)_i$ is spanned over $\F_p$ by operators of the form $x^\gamma \sigma_\alpha$ where $|\gamma| = |\alpha| + i$ (where we use multi-index notation) and thus it suffices to show that such an operator is in $(D^e_S)_0 S_i$. Let $\gamma_0, \gamma_1$ be multi-exponents with $\gamma_0 \in L$ such that $x^\gamma = x^{p^e \gamma_1} x^{\gamma_0}$. Let $\beta \in L$ be such that $x^\alpha x^\beta = x_1^{p^e-1} \cds x_n^{p^e-1}$, and let $\tau = \sigma_{(p^e - 1, \ds, p^e - 1)}$ (an operator of degree $-r(p^e-1)$). Observe that
	$$x^\gamma \sigma_\alpha = x^\gamma \tau x^\beta = x^{\gamma_0} \tau x^{p^e \gamma_1} x^\beta.$$
	We claim that the degree of $x^{p^e \gamma_1} x^\beta$ is at least $i$, i.e. that $p^e|\gamma_1| + |\beta| \geq i$. This follows because
	\begin{align*}
	i + |\alpha| & = p^e |\gamma_1| + |\gamma_0| \\
		& \geq p^e |\gamma_1| + r(p^e -1) \\
		& = p^e |\gamma_1| + |\beta| + |\alpha|.
	\end{align*}
	From the claim we conclude that there exist monomials $x^{\eta}$ and $x^\nu$ with $|\nu| = i$ and $x^{p^e \gamma_1} x^\beta = x^\eta x^\nu$. By the above equalities, we have $x^\gamma \sigma_\alpha = x^{\gamma_0} \tau x^\mu x^\nu$. The degree of the operator $x^{\gamma_0} \tau x^\mu$ is zero and thus $x^\gamma \sigma_\alpha \in (D^e_S)_0 S_i$.  
	
	Part (c) follows by combining (a) and (b) and the observation $I^i = S_{\geq i}$.
\end{proof}

\pagebreak

\subsection{The ring $\cC_R$ and test ideals}

Let $R$ be a regular $F$-finite ring of characteristic $p>0$. We denote by $F: R \to R$ the Frobenius endomorphism on $R$ and, given an integer $e > 0$, we denote by $F^e$ its $e$-th iterate. We define $F^e_*: \Mod(R) \to \Mod(R)$ to be the functor that restricts scalars via $F^e$. The $R$-module $F^e_* R$ is then equal to $R$ as an abelian group and we will denote an element $r \in R$ as $F^e_* r$ when viewed as an element of $F^e_* R$. In this way, the $R$-module action on $F^e_* R$ is given by $s \cdot F^e_* r = F^e_* (s^{p^e} r)$ for all $s, r \in R$. Given an ideal $I \sq R$ and $e > 0$ we define an ideal $I^{[p^e]} := (f^{p^e}: f \in I)$.
We begin by defining the ring $\cC_R$, which will act on $R$. This definition is taken from \cite[Def. 2.2, Example 2.3]{Bli13}.
\begin{definition}
	We denote $\cC^e_R := \Hom_R(F^e_* R , R)$. The algebra $\cC_R$ of Cartier linear operators on $R$ is given by $\cC_R := \bigoplus_{e \geq 0} \cC^e_R$ as an abelian group with multiplication defined as follows: if $\alpha \in \cC^e$ and $\beta \in \cC^d$ then $\alpha \cdot \beta := \alpha \circ F^e_* \beta \in \cC_R^{e + d}.$
\end{definition}
This makes $\cC_R$ into a noncommutative ring, which acts on $R$ on the left by $\phi \cdot f := \phi(F^e_* f) \text{ for } \phi \in \cC_R^e$ and $f \in R.$ In particular, given an ideal $I \sq R$ and an integer $e \geq 0$, $\cC^e_R \cdot I$ is the ideal generated by $\{\phi(F^e_* r) : \phi \in \cC^e_R, r \in I\}$. These ideals have been considered in the past with varying notation: see, for example, \cite{AMHNB},  \cite{AMBL} (where $\cC^e_R \cdot f$ is denoted $I_e(f)$) and \cite{BMSm-hyp} (where $\cC^e_R \cdot I$ is denoted $I^{[1/p^e]}$, which is also the notation used by \Mustata \ in \cite{Mustata2009}).

Recall that in our setting the ring $D_R$ of differential operators on $R$ is given by $D_R = \bigcup_{e \geq 0} D^e_R$ where $D^e_R := \End_{R^{p^e}}(R)$. Our next lemma relates the actions of $D^e_R$ and $\cC^e_R$. It is well-known to experts and parts of it are proved in the aforementioned references. We include a proof for completeness.
\begin{lemma} \label{caralgdiffop} \label{cartier-containment-lemma}
	Let $I,J \sq R$ be ideals and $e \geq 0$ be an integer. Then:
	\begin{enumerate}[(a)]
		\item We have $D^e_R \cdot I = (\cC^e_R \cdot I)^{[p^e]}.$
		
		\item One has $\cC^e_R \cdot I \sq J$ if and only if $I\sq J^{[p^e]}$. 
		
		\item One has $D^e_R \cdot I = D^e_R \cdot J$ if and only if $\cC^e_R \cdot I = \cC^e_R \cdot J$. 
	\end{enumerate}

\end{lemma}
\begin{proof}
	We begin with (a). Let $\t{D}^e_R := \End_R(F^e_* R)$ and observe that the natural identification of $R$ and $F^e_*R$ identifies $D^e_R$ and $\t{D}^e_R$. Because $R$ is regular and $F$-finite, Kunz's theorem gives that $F^e_* R$ is locally free of finite rank. It follows that the map $F^e_*R \otimes_R \cC^e_R \to \t{D}^e_R$ given by 
	$$F^e_* r \otimes \phi \to [F^e_*s \mapsto F^e_* \big( r \phi(F^e_*s)^{p^e} \big)]$$
	is an isomorphism. We thus get $F^e_* (D^e_R \cdot I) = \t{D}^e_R(F^e_* I) = F^e_*( (\cC^e_R \cdot I)^{[p^e]} )$, which proves the statement.
	
	For (b), the direction $(\Leftarrow)$ is clear. For $(\Rightarrow)$ observe that $I \sq D^e_R \cdot I$ and $D^e_R \cdot I = (\cC^e_R \cdot I)^{[p^e]}$ by part (a). 
	
	For (c), the direction $(\Leftarrow)$ follows from part (a) so let us prove $(\Rightarrow)$. If $D^e_R \cdot I = D^e_R \cdot J$ then, by part (a), $(\cC^e_R \cdot I)^{[p^e]} = (\cC^e_R \cdot J)^{[p^e]}$. By applying a splitting of the Frobenius map $R \to F^e_* R$ to both sides of this equality, one gets $\cC_R^e \cdot I = \cC_R^e \cdot J$
\end{proof}
We now define the test ideal. As we will only be concerned with the case where $R$ is regular we will take the following description as our definition.
\begin{definition}
	Let $\fa \sq R$ be an ideal and $\lambda \in \R_{>0}$. Then the test ideal for $\fa$ with exponent $\lambda$ is 
	$$\tau(\fa^\lambda) := \bigcup_{e = 0}^\infty \cC^e_R \cdot \fa^{\lceil \lambda p^e \rceil}.$$
\end{definition}

We observe that the union on the right-hand side is an increasing union of ideals and that, since $R$ is noetherian, $\tau(\fa^\lambda) = \cC_R^d \cdot \fa^{\lceil \lambda p^d \rceil}$ for some $d$ large enough. If $\lambda < \mu$ it follows from the definition that $\tau(\fa^\lambda) \sq \tau(\fa^\mu)$. Moreover, given $\lambda$ there exists some $\epsilon > 0$ such that $\tau(\fa^\lambda) = \tau(\fa^\mu)$ for all $\mu \in [\lambda, \lambda + \epsilon]$ \cite[Cor. 2.16]{BMSm2008}. 
\begin{definition} \label{def-Fjn}
We say $\lambda$ is an $F$-jumping number for $\fa$ if for all $\epsilon > 0$ we have $\tau(\fa^{\lambda - \epsilon}) \neq \tau(\fa^{\lambda})$.
\end{definition}
The set of $F$-jumping numbers forms a discrete subset in $\R_{>0}$ and all $F$-jumping numbers are rational (see \cite[Thm. 3.1]{BMSm2008} for the case where $R$ is essentially of finite type over an $F$-finite field, and \cite[Thm. B]{SchTucTI} for the more general statement).

We will repeatedly use the following standard result about test ideals. We include a proof as it is straightforward with our assumptions and definitions.
\begin{proposition}[Skoda's theorem] \label{brsk-prop}
	 If $\fa$ is generated by $r$ elements and $\lambda \geq r$ then:
	\begin{enumerate}[(a)]
		\item For all $e \geq 0$ we have $\cC_R^e \cdot \fa^{\lceil p^e \lambda \rceil} = \fa \ \cC_R^e \fa^{\lceil p^e(\lambda - 1) \rceil}.$
		\item We have $\tau(\fa^{\lambda}) = \fa \ \tau(\fa^{\lambda -1}).$
		\item If $\lambda \geq r$ is an $F$-jumping number for $\fa$ then so is $\lambda - r$. 
	\end{enumerate}
\end{proposition} 
\begin{proof}
	If $s \geq rp^e$ then we have $\fa^s = \fa^{[p^e]} \fa^{s - p^e}$. Using this we get
	$$\cC_R^e \cdot \fa^{\lceil \lambda p^e \rceil} = \cC_R^e \cdot \fa^{[p^e]} \fa^{\lceil (\lambda - 1)p^e \rceil} = \fa \ \cC_R^e \cdot \fa^{\lceil(\lambda - 1)p^e \rceil },$$
	which gives (a). Statements (b) and (c) follow.
\end{proof}
\section{The multi-eigenspace decomposition of the modules $N^e_\fa$} \label{scn-multi-eigen-dec}

In this section we follow \Mustata's approach from \cite{Mustata2009} to develop a Bernstein-Sato theory in positive characteristic. This technique relies on the construction of certain modules $N^e_\fa$ which carry an action of certain operators $s_{p^i}$, and subsequent analysis of this action.

Fix a nonzero ideal $\fa \sq R$ and let us choose generators $f_1, \ds, f_r$ for $\fa$. To this data we will associate some modules $N^e_\fa$, where $e \geq 1$ is an integer.

Let us first set up some notation. We let $t = (t_1, \ds, t_r)$ be a set of variables and by $R[t]$ we denote the polynomial ring $R[t] := R[t_1, \ds, t_r]$. Let $J := (f_1 - t_1, \ds, f_r - t_r) \sq R[t]$, which is the ideal of the graph $X \to X \times \A^r$ of the functions $f_1, \ds, f_r$, where $X = \Spec R$. 

We consider the local cohomology module $H^r_J R[t]$. Via the \v{C}ech complex on the generators $(f_i - t_i)$ for $J$ we write this module as $H^r_J R[t] = \bigoplus_{\nu \in \N^r}^\infty R \delta_\nu$ where, if $\nu = (\nu_1, \ds, \nu_r)$ then $\delta_\nu$ denotes the class of $(f-t)^{-\nu}$ (recall we use multi-index notation). We denote $\delta := \delta_{(1,1, \ds, 1)}$. We equip $D_{R[t]}$ with the $V$-filtration induced by the ideal $(t) = (t_1, \ds, t_r)$. 
	
Given $e > 0$ we define
$$N^e_\fa := \frac{V^0D^e_{R[t]} \cdot \delta}{V^1D^e_{R[t]} \cdot \delta}.$$
Note that the construction of this module depends on the original choice of generators $(f_1, \ds, f_r)$ for $\fa$.

\begin{remark}
	In characteristic zero one considers the module $V^0 D_{R[t]} \cdot \delta / V^1  D_{R[t]} \cdot \delta$ \cite{BMSa}. Observe that in characteristic $p > 0$ we have maps $N^e_\fa \to N^{e+1}_\fa$ and that $\varinjlim_{e} N^e_\fa = V^0 D_{R[t]} \cdot \delta / V^1 D_{R[t]} \cdot \delta$. We thus think of the $N^e_\fa$ as providing a filtration of the module that comes from characteristic zero. We will explore this limiting process further in Sections \ref{scn-Ap-and-modules} and \ref{scn-A-module-N}. 
\end{remark}

As mentioned, a key ingredient in the construction will be to consider the action of certain operators $s_{p^i}$ on the modules $N^e_\fa$. We will next define these operators and describe their basic properties.

\pagebreak

\subsection{The operators $s_m$}

Let $R$ be an $A$-algebra and consider a polynomial ring $R[t] := R[t_1, \ds, t_r]$ over $R$. Given a multi-exponent $\ul{a} \in \N_0^r$ we denote by $\partial_t^{[\ul{a}]}$ the ``divided power differential operator": $\partial_{t}^{[\ul{a}]}$ is the unique $R$-linear operator that acts on a monomial $t^{\ul{k}}$ by $\partial_t^{[\ul{{a}}]} \cdot t^{\ul{k}} = {\ul{k} \choose \ul{a}} t^{\ul{k} - \ul{a}}$. From the definition one easily checks that $\partial_t^{[\ul{a}]} \partial_t^{[\ul{b}]} = {\ul{a} + \ul{b} \choose \ul{a}} \partial_t^{[\ul{a} + \ul{b}]}$. If $A$ is a field of characteristic zero one has $\partial_t^{[\ul{a}]} = \frac{1}{\ul{a}!} \partial_t^{\ul{a}}$.

For every positive integer $m \in \N$ we define the $A$-linear differential operator 
$$s_m := (-1)^m \sum_{\substack{\ul{a} \in \N_0^r \\ |\ul{a}| = m}} \partial_t^{[\ul{a}]} t^{\ul{a}}.$$
The operators $s_m$ also depend on $R$ and $r$ but we suppress these from the notation as they will always be clear from context. Even though we defined them for all $m>0$, in a characteristic-$p$ setup we will only be interested in operators $s_m$ where $m = p^i$ for some $i$.

Given an integer $n \geq 0$ we denote by $n_i$ the $i$-th digit in the base-$p$ expansion of $n$. That is to say, the $n_i$ are integers with $0 \leq n_i < p$ and such that $n = n_0 + pn_1 + p^2 n_2 + \cds$.

\begin{proposition} \label{prop-s_m-properties} \label{spi-combined-prop}
	The operators $s_m$ on $R[t_1, \ds, t_r]$ satisfy the following properties:
	\begin{enumerate}[(a)]
		\item They are pairwise commuting, i.e. $s_m s_l = s_l s_m$ for all $m, l \in \N$. 
		\item We have $m! s_m = s_1 (s_1 - 1) \cds (s_1 - m + 1)$. 
		\item If $\ul{a} \in \N_0^r$ is a multi-exponent then for all $m > 0$ we have
		$$s_m \cdot t^{ \ul{a}} = (-1)^m {|\ul{a}| + r + m - 1\choose m } t^{\ul{a}}.$$
	\end{enumerate}
If $R$ has characteristic $p$ then, moreover,
\begin{enumerate}[resume*]
	\item Fix an integer $e > 0$ and a multi-exponent $\ul{a} \in \N_0^r$.  Then for all $i < e$ we have 
		$$s_{p^i} \cdot t^{(p^e -1 )\mathbbm{1} - \ul{a}} =|\ul{a}|_i \ t^{\ul{a}}.$$
	\item Given integers $0 \leq i < e$, the operator $s_{p^i}$ is in $D^e_{R[t]}$. 
	\item The operators $\{s_{p^i} : i = 0 , 1, \ds \}$ satisfy $s_{p^i}^p = s_{p^i}$ or, equivalently, $\prod_{j \in \F_p}(s_{p^i} - j) = 0$.
	\item Given $e > 0$ and a module $M$ over $\F_p[s_{p^0}, s_{p^1}, \ds, s_{p^{e-1}}]$ the module $M$ splits as a direct sum
	\begin{equation*} 
		M = \bigoplus_{\alpha \in \F_p^e} M_\alpha \label{multi-eigen-dec-eqn}
	\end{equation*}
	where for all $\alpha = (\alpha_0, \ds, \alpha_{e-1}) \in F_p^e$ and all $0 \leq i < e$ the operator $s_{p^i}$ acts on $M_\alpha$ by the scalar $\alpha_i$.
\end{enumerate}
\end{proposition}

\begin{proof}
	We first notice that if $\ul{a} \in \N_0^r$ is a multi-exponent then it is clear from the definition that $s_m$ acts on the monomial $t^{\ul{a}}$ by an integer scalar (in fact, (c) will say what scalar it is). Statement (a) follows.
	
	We claim that $m s_m = s_{m - 1} (s_1 - m + 1)$, which will prove (b). One first checks that 
	$$\partial_{t_i}^{[a]} t_i^a \partial_i t_i = (a+1) \partial_{t_i}^{[a + 1]} t_i^{a+1} - a \partial_{t_i}^{[a]} t^a,$$
	and the claim then follows from the following computation:	
	\begin{align*}
	(-1)^m s_{m -1} s_1  & = \bigg( \sum_{|\ul{a}| = m - 1} \partial_t^{[\ul{a}]} t^{\ul{a}} \bigg) \bigg( \sum_{i = 1}^r \partial_{t_i} t_i \bigg)\\
		& = \sum_{|\ul{a}| = m -1} \sum_{i=1}^r \partial_{t_1}^{[a_1]} t_1^{a_1} \cds \bigg((a_i + 1) \partial_{t_i}^{[a_i + 1]} t^{a_i + 1} - a_i \partial_{t_i}^{[a_i]}t^{a_i} \bigg) \cds \partial_{t_r}^{[a_r]} t_r^{[a_r]} \\
		& = \sum_{|\ul{a}| = m -1} \sum_{i=1}^r (a_i + 1) \partial_{t_1}^{[a_1]} t_1^{a_1} \cds \partial_{t_i}^{[a_i + 1]} t^{a_i + 1}  \cds \partial_{t_r}^{[a_r]} t_r^{[a_r]}  \hspace*{10pt} + \\   
		& \hspace*{75pt} +  \sum_{|\ul{a}| = m -1} \sum_{i=1}^r  (-a_i) \partial_{t_1}^{[a_1]} t_1^{a_1} \cds  \partial_{t_r}^{[a_r]} t_r^{[a_r]} \\
		& = m (-1)^m s_m - (m-1) (-1)^{m-1} s_{m - 1}.
	\end{align*}
	We now turn to (c). The operator $s_m$ on $R[t_1, \ds, t_r]$ is, by definition, the extension of the operator $s_m$ on $\Z[t_1, \ds, t_r]$. We therefore restrict to the case $A = R = \Z$. On the monomial $t^{\ul{a}}$ the operator $s_1$ acts by the scalar $- |\ul{a}| - r$ and, from part (b), we conclude that $s_m$ acts by the scalar
	$$\frac{1}{m!} (- |\ul{a}| - r) (- |\ul{a}| - r - 1) \cds (- |\ul{a}| - r - m + 1) = (-1)^m {|\ul{a}| + r + m - 1 \choose m }$$
	as required.
	
	Suppose now that $R$ has characteristic $p>0$, and let us prove (d) From part (c) it follows that $s_{p^i}$ acts on $t^{(p^e - 1) \mathbbm{1} - \ul{a}}$ by the scalar:
	$$- {|(p^e - 1)\mathbbm{1}| - |\ul{a}| + r + p^i - 1 \choose p^i} = - {r p^e - |\ul{a}| + p^i - 1 \choose p^i}.$$
	From Lucas' theorem (Lemma \ref{lucas-thm-lemma} below) it follows that ${m \choose p^i } \equiv m_i \mod p$, where $m_i$ is the $i$-th digit in the base-$p$ expansion of $m$. We conclude that the scalar above is equal to
	$$- (rp^e - |\ul{a}| + p^i - 1)_i = -(rp^e - 1- |\ul{a}|)_i - 1 = - (p - 1- |\ul{a}|_i) - 1 = |\ul{a}|_i$$
	as required. From part (d) we deduce that if $\ul{a} \in \N_0^r$ is a multi-exponent then, for all $i < e$, the operator $s_{p^i}$ acts on monomials $t^{\ul{a}}$ and $t_i^{p^e} t^{\ul{a}}$ by the same scalar. Part (e) follows. Part (f) follows from the fact that the operators $s_{p^i}$ act on monomials $t^{\ul{a}}$ by $\F_p$-scalars. Part (g) follows from (f).
\end{proof}
\begin{lemma}[Lucas' theorem] \label{lucas-thm-lemma}
	Let $m, n \geq 0$ be integers. Then
	$${m \choose n} \equiv \prod_{j = 0}^\infty {m_j \choose n_j} \mod p.$$
\end{lemma}
Observe that since for $j$ large enough one has $m_j = n_j = 0$ and ${0 \choose 0} = 1$ by convention the right-hand side of this congruence has a well-defined value.
\begin{proof}
	One equates the coefficient of $x^n$ on the two sides of the following equality in $\F_p[x]$: $(1+x)^m = \prod_{i = 0}^\infty(1 + x^{p^i})^m$. 
\end{proof}

We refer to the decomposition in Proposition \ref{spi-combined-prop} (g) as the multi-eigenspace decomposition of $M$ with respect to $s_{p^0}, \ds, s_{p^{e-1}}$. The following example of this property will be crucial.
\begin{example} \label{crucialexample2}
	Let $e> 0$ and let us consider the action of $s_{p^0}, \ds, s_{p^{e-1}}$ on $T := R[t] = R[t_1, \ds, t_r]$. For all $\alpha = (\alpha_0, \ds, \alpha_{e - 1}) \in \F_p^e$ the module $T_\alpha$ is generated over $R[t_1^{p^e}, \ds, t_r^{p^e}]$ by the monomials
	$$\{t^{(p^e - 1)\mathbbm{1} - \ul{a}} \ | \ \ul{a} \in \{0, \ds, p^e -1\}^r \text{ and } |\ul{a}| \equiv |\alpha| \text{ mod }  p^e \},$$
	where $|\alpha| = \alpha_0 + p \alpha_1 + \cds + p^{e-1} \alpha_{e-1}$ after identifying each $\alpha_i \in \F_p$ with its unique representative in $\{0,1, \ds, p-1\}$. 
\end{example}

Let us note that the local cohomology module $H^r_J R[t]$ is a $D_{R[t]}$-module (for example, via the \v{C}ech complex) and that $N^e_\fa$ is a module over the ring $V^0 D^e_{R[t]}$. It follows from Proposition \ref{spi-combined-prop} (e), and the observation that the $s_{p^i}$ have degree zero in the $t$'s, that the operators $s_{p^0}, s_{p^1}, \ds, s_{p^{e-1}}$ act on the module $N^e_\fa$. By Proposition \ref{spi-combined-prop} (g) $N^e_\fa$ has an multi-eigenspace decomposition with respect to the operators $s_{p_0}, \ds, s_{p^{e-1}}$. Our next goal will be to describe this decomposition, which achieved in Theorem \ref{thm-summands-of-Ne}.
\subsection{The module $H^r_J R[t]$ in positive characteristic}

In the construction of our module $N_\fa$ at the beginning of Section \ref{scn-multi-eigen-dec} our first step was to consider the module $H^r_J(R[t])$, which we construct via the \v{C}ech complex on the generators $(f_1 - t_1), \ds, (f_r - t_r)$. As above, given $\nu \in \N^r$ we let $\delta_\nu \in H^r_J R[t]$ be the class of $(f_1 - t_1)^{-\nu_1} \cds (f_r - t_r)^{-\nu_r}$. It follows that $H^r_J R[t]$ is a free $R$-module on the $\delta_\nu$. We let $\delta := \delta_{(1,1, \ds, 1)}$. 

The module $H^r_J R[t]$ has a very useful extra structure in positive characteristic -- that of a unit $F$-module over $R[t]$. Let us quickly recall the definition.

We denote by $R[t]^{(e)}$ the $(R[t], R[t])$-bimodule which is $R[t]$ on the left and $F^e_* R[t]$ on the right.

\begin{definition}
	A unit $F$-module over $R[t]$ is an $R[t]$-module $M$ together with an isomorphism $\nu_1: R[t]^{(1)} \otimes_{R[t]} M \xrightarrow{\sim} M$.
\end{definition}

The data of the map $\nu_1$ is equivalent to that of a an additive map $F: M \to M$ satisfying $F(g u) = g^p F(u)$ for all $g \in R[t]$ and $u \in M$, where the equivalence is given by $\nu_1(g \otimes u) = gF(u)$, and $F(u) = \nu_1(1 \otimes u)$. If $M$ is a unit $F$-module over $R[t]$ we can iterate the structure map $\nu_1$ to obtain an isomorphism $\nu_e: R[t]^{(e)} \otimes_{R[t]} M \xrightarrow{\sim} M$. A unit $F$-module over $R[t]$ has a natural $D_{R[t]}$-module structure, where an operator $\xi \in D^e_{R[t]}$ acts naturally on the left side of the tensor $R[t]^{(e)} \otimes M$. One can therefore ask what the decomposition from Proposition \ref{spi-combined-prop} (g) looks like. The answer is as follows.

\begin{lemma} \label{unit-F-mod-dec-lemma}
	Let $M$ be a unit $F$-module over $R[t]$. Fix $e > 0$ and let $\alpha = (\alpha_0, \ds, \alpha_{e-1}) \in \F_p^e$. Then 
	$$M_\alpha = \nu_e(T_\alpha \otimes M)$$
	where $T_\alpha$ is as in Example \ref{crucialexample2}.
\end{lemma}
\begin{proof}
From Example \ref{crucialexample2} we have that the multi-eigenspace decomposition for $R[t]^{(e)}$ is $R[t]^{(e)} = \bigoplus_\alpha T_\alpha$. We thus have $M = \nu_e(R^{(e)} \otimes M) = \bigoplus_\alpha \nu_e(T_\alpha \otimes M)$ and, as the operators $s_{p^i}$ act by $s_{p^i} \cdot \nu_e(g \otimes u) = \nu_e((s_{p^i} \cdot g) \otimes u)$, this gives the multi-eigenspace decomposition of $M$. 
\end{proof}
The module $R[t]$ is itself a unit $F$-module, as well as any of its localizations. The cokernel of a map of unit $F$-modules is also a unit $F$-module. By considering the \v{C}ech complex it follows that $H^r_J R[t]$ is a unit $F$-module. The unit $F$-module structure on $H^r_J R[t]$ is given by $\nu_1(g \otimes \delta_\nu) = g \delta_{p \nu}$.

Given an $e>0$ and an $r$-tuple $\ul{a} \in\{0, \ds, p^e -1 \}^r$, set
$$Q^e_{\ul{a}} := \nu_e(t^{(p^e - 1) \mathbbm{1}- \ul{a}} \otimes \delta ).$$
By Lemma \ref{unit-F-mod-dec-lemma}, $Q^e_{\ul{a}}$ is an eigenvector for the multi-eigenvalue $\alpha \in \F_p^e$ where $\alpha_i = |\ul{a}|_i$.

Our first goal will be to show the following.

\begin{proposition} \label{Q-lin-indep}
	Fix $e > 0$. Then the elements $\{Q^e_{\ul{a}} : 0 \leq a_j < p^e\}$ are linearly independent over $R$.
\end{proposition}

Before we begin, we remark that just as in \cite{Mustata2009} it is true that for a fixed $e > 0$ the elements $\nu_e(t^{(p^e - 1)\mathbbm{1} - \ul{a}} \otimes \delta_\nu)$ as $\nu$ ranges through $\{0, \ds, p^e-1\}^r$ form an $R$-basis for the module $H^r_J R[t]$. However we will only require the result above.

We first want to express the elements $Q^e_{\ul{a}}$ in terms of the $R$-basis $\delta_\nu$.

\begin{lemma} \label{Q-decomp}
	Let $e$ and $\ul{a}$ be as above. Let $E_{\ul{a}}$ be the set of $r$-tuples of integers $(i_1, \ds, i_r)$ with $0 \leq i_j \leq p^e - 1- a_j$ for all $j$. We then have
	$$Q^e_{\ul{a}} = \sum_{\ul{i} \in E_{\ul{a}}} u_{(e,\ul{a}, \ul{i})} f^{\ul{i}} \delta_{\mathbbm{1} + \ul{a} + \ul{i}},$$
	where
	$$u_{(e,\ul{a}, \ul{i})} := (-1)^{|\ul{a}| + |\ul{i}|} {(p^e - 1) \mathbbm{1} - \ul{a} \choose \ul{i}}.$$
\end{lemma}
\begin{proof}
	We have
	\begin{align*}
	Q^e_{\ul{a}} & = \nu_e \bigg( (f - (f-t))^{(p^e - 1)\mathbbm{1}- \ul{a}} \otimes \delta \bigg) \\
	& = \sum_{\ul{i} \in E_{\ul{a}}} (-1)^{|\ul{a}| + |\ul{i}|} {(p^e - 1)\mathbbm{1}- \ul{a} \choose \ul{i}} f^{\ul{i}} \nu_e\bigg((f-t)^{(p^e - 1)\mathbbm{1}- \ul{a} - \ul{i}} \otimes \delta\bigg), \\	
	\end{align*}
	where the second equality follows from the multi-variate binomial theorem. The proof is complete once we observe the fact that $\nu_e((f-t)^{\ul{b}} \otimes \delta) = \delta_{p^e \mathbbm{1} - \ul{b}}$. 
\end{proof}

Given an integer $n\geq 0$ let $F_n$ be the $R$-module spanned by $\delta_{\nu}$ for which $\sum_j \nu_j > n$.

\begin{lemma} \label{q-linindep-lemma}
	Suppose $\{q_{\ul{a}}: \ul{a} \in \N^r\}$ is a collection of elements satisfying
	$$q_{\ul{a}} \equiv \delta_{\ul{a}} \ \ \ \mod F_{a_1 + \cds+ a_r}$$
	for all $\ul{a} \in \N^r$. Then the $\{q_{\ul{a}}\}$ are linearly independent over $R$.
\end{lemma}

\begin{proof}
	Suppose we had a relation
	$$r_1 q_{\ul{a}_1} + \cds + r_k q_{\ul{a}_k} = 0$$
	and assume for a contradiction that all $r_i \neq 0$. Let $a_{ij}$ be the entries of $\ul{a}_i$, i.e. $\ul{a}_i = (a_{i1}, \ds, a_{ir})$. Let
	$$n := \min\{a_{i1} + \cds + a_{ir}: i = 1, \ds, k\}$$
	and by relabeling assume that $n = a_{11} + \cds + a_{1r}  = \cds = a_{s1} + \cds + a_{sr}$ and $n < a_{i1} + \cds + a_{ir}$ for all $i > s$. By considering the relation modulo $F_n$ we obtain
	$$r_1 \delta_{\ul{a}_1} + \cds + r_s \delta_{\ul{a}_s} \equiv 0 \mod F_n.$$
	Recall that $H^r_J R[t]$ is a free $R$-module on the basis $\delta_{\ul{a}}$. It follows that $H^r_J R[t]/F_n$ is a free module on the basis $\{\delta_{\ul{a}}: a_1 + \cds+ a_r \leq n\}$, which implies that $r_1 = \cds = r_s = 0$, a contradiction.
\end{proof}

We can now prove Proposition \ref{Q-lin-indep}.

\begin{proof}[Proof of Proposition \ref{Q-lin-indep}]
	For all $\ul{a} \in \N^r$ let $q_{\ul{a}} := (-1)^{|\ul{a}|} Q^e_{\ul{a} - \mathbbm{1}}$. From Lemma \ref{Q-decomp} it follows that $q_{\ul{a}} \equiv \delta_{\ul{a}} \mod F_{a_1 + \cds + a_r}$ and thus by Lemma \ref{q-linindep-lemma} the $q_{a,b}$ are linearly independent. It follows that the $Q^e_{\ul{a}}$ are also linearly independent.
\end{proof}

Finally, we express $\delta$ in terms of the $Q^e_{\ul{a}}$.

\begin{lemma} \label{delta_dec}
	For all $e > 0$ we have
	$$\delta = \sum_{\ul{a} \in \{0, \ds, p^e-1\}^r} u_{\ul{a}} f^{\ul{a}} Q^e_{\ul{a}}$$
	where the $u_{\ul{a}}$ are units in $\F_p$.
\end{lemma}

\begin{proof}
	We have
	\begin{align*}
	\delta & = \nu_e((f-t)^{(p^e - 1)\mathbbm{1}} \otimes \delta) \\
	& = \nu_e \lp \sum_{\ul{a} \in \{0, \ds, p^e - 1\}^r} (-1)^{|\ul{a}|} {(p^e - 1)\mathbbm{1} \choose \ul{a}} f^{\ul{a}} t^{(p^e -1)\mathbbm{1} - \ul{a}} \otimes \delta \rp \\
	& = \sum_{\ul{a} \in \{0, \ds, p^e -1\}^r} u_{\ul{a}} f^{\ul{a}} Q^e_{\ul{a}}.
	\end{align*}
	where we have set
	$$u_{\ul{a}} = (-1)^{|\ul{a}|} {(p^e -1 )\mathbbm{1} \choose \ul{a}}.$$
	Finally, observe that given an integer $b$ with $0 \leq b < p^e$ by Lemma \ref{lucas-thm-lemma} (Lucas' theorem) we have the following equality in $\F_p$:
	$${p^e - 1 \choose b} = \prod_{i = 1}^{e-1} {p-1 \choose b_i}.$$
	It follows that ${p^e - 1 \choose b}$ is nonzero in $\F_p$, and this shows that $u_{\ul{a}}$ is nonzero in $\F_p$. This completes the proof.
\end{proof}

\subsection{Description of the modules $N^e_\fa$.}

We are finally in a position to describe the eigenspaces of $N^e_\fa$. Let us first introduce some notation: given $\alpha = (\alpha_0, \ds, \alpha_{e-1}) \in \F_p^e$ we denote by $|\alpha|$ the integer $|\alpha| := \alpha_0 + p \alpha_1 + \cds + p^{e-1} \alpha_{e-1}$, where we have identified each $\alpha_i \in \F_p$ with its unique representative in $\{0, \ds, p-1\}$. Observe that if $\alpha\in \F_p^e$ then $0 \leq |\alpha| < p^e$. 
\begin{theorem} \label{thm-summands-of-Ne}
	Let $e > 0$ be an integer and $\alpha = (\alpha_0, \ds, \alpha_{e-1}) \in \F_p^e$. 
	\begin{enumerate}[(a)]
		\item The multi-eigenspace $(N^e_\fa)_\alpha$ is then a direct sum of the modules in the set
		$$\bigg\{ \frac{D^e_R \cdot \fa^{|\alpha| + sp^e}}{D^e_R \cdot \fa^{|\alpha| + sp^e+1}} \ \bigg| \ s = 0, 1, \ds, r-1 \bigg\},$$
		and in turn each such module occurs in the direct sum.
		\item The module $N^e_\alpha$ is zero if and only if for all $s \in \{0, 1, \ds , r-1\}$ we have
		$$\cC^e_R \cdot \fa^{|\alpha| + sp^e} = \cC^e_R \cdot \fa^{|\alpha| + sp^e + 1}.$$
	\end{enumerate}
\end{theorem}
Moreover, we are able to give a complete description of the module $(N^e_\fa)_\alpha$: see Remark \ref{rmk-Ne-direct-sum}. We will also show later that the set $\{\alpha \in \F_p^e : N^e_\alpha \neq 0\}$ is independent of the initial choice of generators (Corollary \ref{cor-trun-bs-root-indep}).

Observe that part (b) of the theorem follows from part(a) and Lemma \ref{caralgdiffop}. We begin working towards the proof of (a).

Recall that we give $R[t]$ an $\N_0$-grading by $\deg t_i= 1$, and that the rings $D^e_{R[t]}$ inherit $\Z$-gradings. We denote $(D^e_{R[t]})_d$ the degree $d$ piece of this grading. Let $I := (t_1, \ds, t_r)$. 

\begin{lemma} \label{lemman1}
	We have
	$$N^e_\fa = \frac{(D^e_{R[t]})_0 \cdot \delta}{(D^e_{R[t]})_0 \cdot \fa \delta}.$$
\end{lemma}
\begin{proof}
	Recall that, by Lemma \ref{V-filt-lemma}, we have $V^0 D^e_{R[t]} = (D^e_{R[t]})_{\geq 0} = (D^e_{R[t]})_0 + (D^e_{R[t]})_0I $ and $V^1 D^e_{R[t]} = (D^e_{R[t]})_0 I$. Note also that since $(f_j - t_j) \delta = 0$ for all $j$ we have $I \delta = \fa \delta$. It follows that 
	\begin{align*}
	N^e_\fa & := \frac{V^0D^e_{R[t]} \cdot \delta}{V^1 D^e_{R[t]} \cdot \delta} \\
	& = \frac{(D^e_{R[t]})_0 \cdot \delta + (D^e_{R[t]})_0 \cdot \fa \delta}{(D^e_{R[t]})_0 \cdot \fa \delta} \\
	& = \frac{(D^e_{R[t]})_0 \cdot \delta}{(D^e_{R[t]})_0 \cdot \fa \delta}
	\end{align*}
	as required.
\end{proof}

We next describe the ring $(D^e_{R[t]})_0$. Let $P = k[t_1, \ds, t_r]$ where $\deg t_i = 1$ for all $i$, so that $R[t] = R \otimes_k P$. As already mentioned in the proof of Lemma \ref{V-filt-lemma} there is a natural isomorphism $D^e_R \otimes_k D^e_P \xrightarrow{\sim} D^e_{R[t]}$ which identifies $D^e_R \otimes (D^e_P)_0 \cong (D^e_{R[t]})_0$. Recall each element of $D^e_P$ is uniquely determined by its action on monomials $t^{\ul{i}}$ with $\ul{i} \in \{0, \ds, p^e - 1\}^r$ or, equivalently, its action on monomials $t^{(p^e - 1)\mathbbm{1} - \ul{i}}$ with $\ul{i} \in \{0, \ds, p^e -1 \}^r$.

Given $\ul{a} \in \{0, \ds, p^e -1\}^r$ and $\ul{c} \in \Z_{< p^e}^r$ we denote by $\sigma_{\ul{a} \to \ul{c}}$ the unique $P^{p^e}$-linear operator on $P$ with the property
$$\sigma_{\ul{a} \to \ul{c}}  \cdot t^{(p^e -1 )\mathbbm{1} - \ul{i}} = \begin{cases}
t^{(p^e - 1)\mathbbm{1} - \ul{c}} \text{ if } \ul{i} = \ul{a} \\
0 \text{ otherwise}. 
\end{cases}$$

Note that we allow the $c_j$ to be negative but we imposte $c_j < p^e$. 

Observe that $D^e_P$ is spanned by the $\sigma_{\ul{a} \to \ul{c}}$ as a $k$-vector space. Since the degree of $\sigma_{\ul{a} \to \ul{c}}$ is $|\ul{a}| - |\ul{c}|$, we see that $(D^e_P)_0$ is therefore generated over $k$ by the operators $\sigma_{\ul{a} \to \ul{c}}$ for which $|\ul{a}| = |\ul{c}|$. We conclude the following.

\begin{lemma} \label{lemman2}
	We have
	$$(D^e_{R[t]})_0 \cong D^e_R \otimes_k (D^e_P)_0 = \bigoplus_{\ul{a}, \ul{c}} D^e_R \cdot \sigma_{\ul{a} \to \ul{c}}$$
	where the sum ranges over all $\ul{a} \in \{0, \ds, p^e -1\}^r$ and $\ul{c} \in \Z_{<p^e}^r$ for which $|\ul{a}| = |\ul{c}|$. 
\end{lemma}

Let us observe how the operators $\sigma_{\ul{a} \to \ul{c}}$ act on the elements $Q^e_{\ul{i}}$.

\begin{lemma} \label{sigma_lemma}
	Let $\ul{a}, \ul{c}', \ul{i} \in \{0, \ds, p^e -1\}^r$ and $\ul{c}'' \in \Z_{\geq 0}^r$. We  then have:
		$$\sigma_{\ul{a} \to \ul{c}' - p^e \ul{c}''} \cdot Q^e_{\ul{i}} = \begin{cases}
		f^{p^e \ul{c}''} Q^e_{\ul{c}'} \text{ if } \ul{i} = \ul{a} \\
		0 \text{ otherwise. }
		\end{cases}$$
\end{lemma}
\begin{proof}
	Recall that our operators act via the unit $F$-module structure. That is, if $\xi \in D^e{R[t]}$ then $\xi \cdot \nu_e(g \otimes u) = \nu_e (\xi \cdot g \otimes u)$. Let us first consider the case where $\ul{c}'' = 0$. It follows that
	$$\sigma_{\ul{a} \to \ul{c}'} Q^e_{\ul{i}} = \nu_e (\sigma_{\ul{a} \to \ul{c}'} \cdot t^{(p^e - 1)\mathbbm{1} - \ul{i}} \otimes \delta)$$
	which gives the statement. To complete the proof note that 
	$$\sigma_{\ul{a} \to \ul{c}' - p^e \ul{c}''} = t^{p^e \ul{c}''} \sigma_{\ul{a} \to \ul{c}'}$$
	and that
	\begin{align*}
	t^{p^e \ul{c}''} Q^e_{\ul{c}'} & = \nu_e(t^{(p^e -1)\mathbbm{1} - \ul{c}'} \otimes t^{\ul{c}''} \delta) \\
	& = \nu_e(t^{(p^e -1) \mathbbm{1} - \ul{c}'} \otimes f^{\ul{c}''} \delta) \\
	& = f^{p^e \ul{c}''} \nu_e(t^{(p^e -1)\mathbbm{1} - \ul{c}'} \otimes \delta) \\
	& = f^{p^e \ul{c}''} Q^e_{\ul{c}'}. \qedhere
	\end{align*}
\end{proof}
With this one last lemma, the proof of the theorem is almost complete.

\begin{lemma} \label{lemman3}
	We have
	$$R (D^e_P)_0 \cdot \delta = \bigoplus_{\ul{a} \in \{0, \ds, p^e -1\}^r} \fa^{|\ul{a}| } \ Q^e_{\ul{a}}.$$
\end{lemma}
\begin{proof}
	First of all, the fact that the sum on the right-hand side is direct follows from Proposition \ref{Q-lin-indep}. 
	
	For the inclusion $(\sq)$ notice that the left-hand side is generated over $R$ by $\sigma_{\ul{a} \to \ul{c}} \cdot \delta$ where $\ul{a} \in \{0, \ds, p^e -1\}^r$ and $\ul{c} \in \Z_{<p^e}^r$ are such that $|\ul{a}| = |\ul{c}|$. Let $\ul{c}' \in \{0, \ds, p^e -1\}$ and $\ul{c}'' \in \Z_{\geq 0}^r$ be such that $\ul{c} = -p^e \ul{c}'' + \ul{c}'$. By Lemma \ref{sigma_lemma} and Lemma \ref{delta_dec},
	$$\sigma_{\ul{a} \to \ul{c}} \cdot \delta =  u_{\ul{a}} f^{\ul{a} + p^e \ul{c}''} Q^e_{\ul{c}'}.$$
	The equality $|\ul{a}| = |\ul{c}|$ implies $|\ul{a} + p^e \ul{c}''| = |\ul{c}'|$, and $\sigma_{\ul{a} \to \ul{c}} \cdot \delta \in \fa^{|\ul{c}'|} Q^e_{\ul{c}'}.$
	
	For the inclusion $(\supseteq)$ observe that the right-hand side is generated over $R$ by $f^{\ul{c}} Q^e_{\ul{a}}$ where $\ul{c} \in \Z_{\geq 0}^r$, $\ul{a} \in \{0, \ds p^e -1\}^r$ and $|\ul{c}| = |\ul{a}|$. Given such $\ul{a}$ and $\ul{c}$ choose $\ul{c}' \in \{0, \ds, p^e -1\}^r$ and $\ul{c}'' \in \Z_{\geq 0}^r$ such that $\ul{c} = p^e \ul{c}'' + \ul{c}'$. Lemmas \ref{sigma_lemma} and \ref{delta_dec} once again give
	\begin{align*}
	f^{\ul{c}} Q^e_{\ul{a}} & = \sigma_{ \ul{c}' \to \ul{a} - p^e \ul{c}''} \cdot f^{\ul{c}'} Q^e_{\ul{c}'} \\
	& = u_{\ul{c}'}^{-1} \sigma_{\ul{c}' \to \ul{a} - p^e \ul{c}''} \cdot \delta.
	\end{align*}
	To finish, we observe that the condition $|\ul{a}| = |\ul{c}|$ gives $|\ul{c}'| = |\ul{a} - p^e \ul{c}''|$ and therefore $\sigma_{\ul{c}' \to \ul{a} - p^e \ul{c}''} \in (D^e_P)_0$.
\end{proof}

\begin{proof}[Proof of Theorem \ref{thm-summands-of-Ne}]
	As mentioned, part (b) follows from part (a) and Lemma \ref{caralgdiffop}. We now prove part (a) By combining our previous lemmas we obtain:
	\begin{align*}
	N^e_\fa & = \frac{(D^e_{R[t]})_0 \cdot \delta}{(D^e_{R[t]})_0 \cdot \fa \delta} & \text{by Lemma \ref{lemman1}} \\
	& = \frac{D^e_R (D^e_P)_0 \cdot \delta}{D^e_R (D^e_P)_0 \cdot \fa \delta} & \text{by Lemma \ref{lemman2}} \\
	& = \frac{   \bigoplus_{\ul{a}} D^e_R \cdot \fa^{|\ul{a}|}Q^e_{\ul{a}}   } {\bigoplus_{\ul{a}} D^e_R \cdot \fa^{|\ul{a}| + 1}Q^e_{\ul{a}}} & \text{by Lemma \ref{lemman3}} \\
	& \cong \bigoplus_{\ul{a} \in \{0, \ds, p^e -1\}^r} \frac{D^e_R \cdot \fa^{|\ul{a}|}}{D^e_R \cdot \fa^{|\ul{a}|+1}}\t{Q}^e_{\ul{a}},
	\end{align*}
	where the operator $s_{p^i}$ acts on $\t{Q}^e_{\ul{a}}$ by the scalar $|\ul{a}|_i$ for all $i = 0, 1, \ds, e -1$ (c.f. comment above Proposition \ref{Q-lin-indep}). The result follows.	
\end{proof}

\begin{remark} \label{rmk-Ne-direct-sum} 
	In fact, it follows from the proof that 
	$$N^e_\fa \cong \bigoplus_{\ul{a} \in \{0, \ds, p^e -1\}^r} \frac{D^e_R \cdot \fa^{|\ul{a}|}}{D^e_R \cdot \fa^{|\ul{a}|+1}}\t{Q}^e_{\ul{a}},$$
	where $\t{Q}^e_{\ul{a}}$ is the image of $Q^e_{\ul{a}}$ in the quotient. It follows that 
	$$(N^e_\fa)_\alpha \cong \bigoplus_{\substack{\ul{a} \in \{0, \ds, p^e -1\}^r\\|\ul{a}| \equiv |\alpha| \mod p^e}} \frac{D^e_R \cdot \fa^{|\ul{a}|}}{D^e_R \cdot \fa^{|\ul{a}|+1}}\t{Q}^e_{\ul{a}}.$$
\end{remark}

We end with a description of the maps $N_\fa^e \to N_\fa^{e+1}$ that will be important later. From the definition of the modules $N_\fa^e$ it follows that this natural map is $V^0D^e_R[t]$-linear (in particular, $D^e_R$-linear) and thus it suffices to describe the image of the $\t{Q}^e_{\ul{a}}$. 

\begin{lemma} \label{Qe-to-Qe+1}
	The natural map $N^e_\fa \to N^{e+1}_\fa$ sends $\t{Q}^e_{\ul{a}}$ to
	$$\sum_{\ul{i} \in \{0, \ds, p-1\}^r} u_{\ul{i}} f^{p^e \ul{i}} \ \t{Q}^{e+1}_{p^e\ul{i} + \ul{a}},$$
	where $u_{\ul{i}}$ is a unit in $\F_p$.
\end{lemma}
\begin{proof}
	First of all observe that $\nu_e(1 \otimes \delta) = \nu_{e+1}((f-t)^{p^e(p-1)} \otimes \delta) = \nu_{e+1}((f^{p^e} - t^{p^e})^{p-1} \otimes \delta).$ We have:
	\begin{align*}
	Q^e_{\ul{a}} & = \nu_e\bigg(t^{(p^e -1)\mathbbm{1} -\ul{a}} \otimes \delta\bigg) \\
		& = \nu_{e+1} \bigg( t^{(p^e - 1)\mathbbm{1}}(f^{p^e} - t^{p^e})^{(p-1) \mathbbm{1}} \otimes \delta \bigg) \\ 
		& = \sum_{\ul{i} \in \{0, \ds, p-1\}^r} (-1)^{|\ul{i}|} {(p-1) \mathbbm{1} \choose \ul{i}} f^{p^e \ul{i}} \  \nu_{e + 1} \bigg( t^{(p^{e + 1} - 1)\mathbbm{1} - (p^e\ul{i} + \ul{a})} \otimes \delta \bigg)\\
		& = \sum_{\ul{i} \in \{0, \ds, p-1\}^r} u_{\ul{i}} f^{p^e \ul{i}} Q^{e+1}_{p^e \ul{i} + \ul{a}},
	\end{align*}
	where we have set
	$$u_{\ul{i}} := (-1)^{|\ul{i}|} {(p-1) \mathbbm{1} \choose \ul{i}}.$$
	This was already proved to be a unit in Lemma \ref{delta_dec}. This completes the proof.
\end{proof}

\begin{remark} \label{Qe-to-Qe+1-rmk}
	Let $\ul{a} \in \{0, \ds, p^e -1\}^r$, $\ul{c} \in \{0, \ds, p-1\}^r$ and set $\ul{b} := \ul{a} + p^e \ul{c} \in \{0, \ds, p^{e+1} -1\}^r$. It follows from Lemma \ref{Qe-to-Qe+1} that the map
	$$\frac{D^e_R \cdot \fa^{|\ul{a}|}}{D^e_R \cdot \fa^{|\ul{a}|+1}}\t{Q}^e_{\ul{a}} \to \frac{D^{e+1}_R \cdot \fa^{|\ul{b}|}}{D^{e+1}_R \cdot \fa^{|\ul{b}| + 1}} \t{Q}^{e+1}_{\ul{b}}$$
	induced by the map $N_\fa^e \to N_\fa^{e+1}$ is, up to a unit, given by multiplication by $f^{p^e \ul{c}}$. In particular, it is $D^e_R$-linear as expected.
\end{remark}

\section{Approximating polynomials in positive characteristic} \label{scn-appr-poly}

As always we fix an $F$-finite regular ring $R$ and an ideal $\fa \sq R$, as well as generators $\fa = (f_1, \ds, f_r)$ for $\fa$. 

Before defining approximating polynomials we discuss some notions that will be important both conceptually and for the relevant proofs.

\subsection{The $\nu$-invariants and $F$-thresholds}

The invariants $\nu^J_\fa(p^e)$ were introduced in \cite{MTW}. We recall the definition, and give a name to the collection of all such invariants.
\begin{definition} \label{def-nu-invt}
	Given a proper ideal $J \sq R$ containing $\fa$ in its radical and an integer $e > 0$ we define $\nu^J_\fa(p^e) := \max\{n \geq 0 : \fa^n \not\sq J^{[p^e]}\}.$ The set $\nu^\bullet_\fa(p^e) := \{ \nu^J_\fa(p^e) \ \big| \ (1) \neq \sqrt{J} \supseteq \fa \}$ is called the set of $\nu$-invariants of level $e$ for $\fa$. 
\end{definition}
In \cite{MTW} it is shown that if one fixes $J$ as above then the sequence $(\nu^J_\fa(p^e)/p^e)_{e = 0}^\infty$ is increasing an bounded. The limit
$$c^J(\fa) := \lim_{e \to \infty} \frac{\nu^J_\fa(p^e)}{p^e}$$
is called the $F$-threshold of $\fa$ with respect to $J$. The set of $F$-thresholds coincides with the set of $F$-jumping numbers \cite[Cor. 2.30]{BMSm2008}.

The following result is well-known to experts.
\begin{proposition} \label{nu-invt-trun-ti-prop}
	Fix an integer $e > 0$. The set of $\nu$-invariants of level $e$ for $\fa$ is given by
	$$\nu^\bullet_\fa(p^e) = \bigg\{n \geq 0 \ \big| \ \cC_R^e \cdot \fa^n \neq \cC_R^e \cdot \fa^{n+1} \bigg\}.$$
\end{proposition}
\begin{proof}
	First suppose that $\cC^e_R \cdot \fa^n = \cC^e_R \cdot \fa^{n+1}$. If $\fa^n \not\sq J^{[p^e]}$ then $\cC^e_R \cdot \fa^n \neq J$ by Lemma \ref{cartier-containment-lemma}. Therefore $\cC^e_R \cdot \fa^{n+1} \not\sq J$ and thus $\fa^{n+1} \not\sq J^{[p^e]}$. Therefore $n \neq \nu^J_\fa(p^e)$. This proves the inclusion $(\sq)$.
	
	For $(\supseteq)$ suppose that $n \geq 0$ is such that $\cC^e_R \cdot \fa^n \neq \cC^e_R \cdot \fa^{n+1}$. Let $J:= \cC^e_R \cdot \fa^{n+1}$. Observe that $J \neq (1)$: otherwise $\cC^e_R \cdot \fa^n = \cC^e_R \cdot \fa^{n+1}$. Given $f \in \fa$ we note that $(f^m) = \cC^e_R \cdot f^{mp^e}$ and $f^{mp^e} \in \fa^{n+1}$ for $m$ large enough, thus $f \in \sqrt{J}$. 
	
	We claim that $n = \nu^J_\fa(p^e)$ and this will complete the proof. First, by Lemma \ref{cartier-containment-lemma} one has $\fa^{n+1} \sq J^{[p^e]}$ and therefore it suffices to show that $\fa^n \not\sq J^{[p^e]}$. But this follows because if $\fa^n \sq J^[p^e]$ then $\cC^e_R \cdot \fa^n \sq J$, and thus $\cC^e_R \cdot \fa^n = \cC^e_R \cdot \fa^{n+1}$, a contradiction. 
\end{proof}
\begin{corollary} \label{nu-invt-dynamics-cor}
	If $n \geq r p^e$ is a $\nu$-invariant of level $e$ then so is $n - p^e$. 
\end{corollary}
\begin{proof}
	By Skoda's theorem (c.f. Proposition \ref{brsk-prop}) we see that whenever $n \geq rp^e$ we have $\cC^e_R \cdot \fa^n = \fa\cC^e_R \cdot \fa^{n - p^e}$ and $\cC^e_R \cdot \fa^{n + 1} = \fa \cC^e_R \cdot \fa^{n + 1 - p^e}$. If $n$ is a $\nu$-invariant, it follows from Proposition \ref{nu-invt-trun-ti-prop} that $\cC^e_R \cdot \fa^n \neq \cC^e_R \cdot \fa^{n+1}$ and thus $\cC^e_R \cdot \fa^{n - p^e} \neq \cC^e_R \cdot \fa^{n -p^e +1}$. Again by Proposition \ref{nu-invt-trun-ti-prop}, $n - p^e$ is a $\nu$-invariant of level $e$.  
\end{proof}
\subsection{Definition and immediate consequences}

We now define approximating polynomials in positive characteristic. This is analogous to the definition in \cite{Mustata2009}. As always we fix a regular $F$-finite ring $R$ and an ideal $\fa \sq R$ with generators $\fa = (f_1, \ds, f_r)$. We let $N^e := N^e_\fa$ be the modules defined in Section \ref{scn-multi-eigen-dec} with this choice of generators and, for $\alpha \in \F_p^e$, $N^e_\alpha := (N^e_\fa)_\alpha$ be the corresponding multi-eigenspace for the action of $s_{p^0}, \ds, s_{p^{e-1}}$ (c.f. Proposition \ref{spi-combined-prop} (g)). We thus drop the subscript $\fa$ from the notation. Recall that given $\alpha = (\alpha_0, \ds, \alpha_{e-1}) \in \F_p^e$ we denote by $|\alpha|$ the integer $|\alpha| := \alpha_0 + p \alpha_1 + \cds + p^{e-1} \alpha_{e-1}$, where we identify each $\alpha_i \in \F_p$ with its unique representative in $\{0, \ds, p-1\}$. With this notation the definition is as follows.

\begin{definition}
	Let $e > 0$ be an integer. The $e$-th approximating polynomial for $\fa$ is given by
	$$a^e_\fa(s) := \prod_{\{\alpha \in \F_p^e \ : \ N^e_\alpha \neq 0\}}  \lp s - \frac{|\alpha|}{p^e} \rp .$$
\end{definition}
\begin{remark}
	In \cite{Mustata2009} these are called Bernstein-Sato polynomials. However, our computations with monomial ideals show that, unlike the Bernstein-Sato roots defined later (c.f. Definition \ref{dfn-BSroot}), these polynomials fail to provide good analogues to the classical Bernstein-Sato polynomials. Nonetheless, approximating polynomials encode information about the $F$-jumping numbers of $\fa$ and a good understanding of approximating polynomials is crucial in our approach to Bernstein-Sato roots.
\end{remark}
Note that, a priori, the approximating polynomial depends on the choice of generator $(f_1, \ds, f_r)$ for $\fa$ (since the construction of $N^e$ does). We will soon show that this is not the case (c.f. Corollary \ref{cor-ae-indep-of-gens}).

Our previous work immediately yields the following result. Recall that $\nu^\bullet_\fa(p^e)$ is the set of $\nu$-invariants of level $e$ for $\fa$ (c.f. Definition \ref{def-nu-invt} and Proposition \ref{nu-invt-trun-ti-prop}).
\begin{theorem} \label{thm-roots-of-be}
	Let $e > 0$ be an integer. Then all roots of $a^e_\fa(s)$ are simple and lie in $[0,1) \cap \Z \frac{1}{p^e} $. Moreover, the set of roots is given by
	$$\bigg\{\mu - \lfloor \mu \rfloor \ \big| \ \mu \in \frac{\nu^\bullet_\fa(p^e)}{p^e}\bigg \}.$$
\end{theorem}
\begin{proof}
	The fact that all roots are simple follows directly from the definition. Moreover, if $n/p^e$ is a root then it follows from Theorem \ref{thm-summands-of-Ne} that $\cC^e_R \cdot \fa^{n + sp^e} \neq \cC^e_R \cdot \fa^{n + sp^e + 1}$ for some $0 \leq s < r$. By Proposition \ref{nu-invt-trun-ti-prop} we conclude that $n + sp^e$ is a $\nu$-invariant of level $e$. If we let $\mu = (n + sp^e)/p^e$ we have $n/p^e = \mu - \lfloor \mu \rfloor$ and thus $n/p^e$ is in the set given above.
	
	Conversely, if $\mu = m/p^e$ where $m \in \nu^\bullet_\fa(p^e)$ we want to show that $\mu - \lfloor \mu \rfloor$ is a root of $a^e_\fa(s)$. From Corollary \ref{nu-invt-dynamics-cor} we may assume that $0 \leq m < rp^e$. Let $0 \leq n < p^e$ and $s \geq 0$ be such that $m = n + sp^e$ and let $\alpha \in \F_p^e$ be the unique vector with $|\alpha| = n$. By Proposition \ref{nu-invt-trun-ti-prop} we have that $\cC^e_R \cdot \fa^{n + sp^e} \neq \cC^e_R \cdot \fa^{n + sp^e + 1}$. From Theorem \ref{thm-summands-of-Ne} it follows that $N^e_\alpha \neq 0$ and thus $\mu - \lfloor \mu \rfloor = n/p^e$ is a root of $a^e_\fa(s)$.
\end{proof}
\begin{corollary} \label{cor-ae-indep-of-gens}
	The approximating polynomial $a^e_\fa(s)$ does not depend on the choice of generators of $\fa$.
\end{corollary}

We are also able to recover \Mustata's result at this stage.
\begin{corollary} [{\cite[Thm. 6.7]{Mustata2009}}] \label{principal-ideal-cor}
	 Suppose that $\fa$ is a principal ideal. Then the roots of $a^e_\fa(s)$ are given by the rational numbers $\frac{\lceil p^e \lambda \rceil - 1}{p^e}$ as $\lambda$ ranges through all $F$-jumping numbers of $\fa$ in $(0,1]$.
\end{corollary}
\begin{proof}
	Since $\fa$ is principal, for all $n, e \geq 0$ we have $\cC_R^e \cdot \fa^n = \tau(\fa^{n/p^e})$. By Proposition \ref{nu-invt-trun-ti-prop}, $n$ is a $\nu$-invariant of level $e$ if and only if there is an $F$-jumping number in the interval $(\frac{n}{p^e}, \frac{n+1}{p^e}]$. The result then follows from Theorem \ref{thm-roots-of-be} together with the observation that $n := \lceil p^e \lambda \rceil -1$ is the unique integer with $\lambda \in (\frac{n}{p^e}, \frac{n+1}{p^e}]$.
\end{proof}
\subsection{Roots of $a^e_\fa(s)$ and $F$-jumping numbers}
Theorem \ref{thm-roots-of-be}, together with the fact that the limit $c^J(\fa) = \lim_{e \to \infty} \nu^J_\fa(p^e)/p^e$ is an $F$-jumping number elucidates that the roots of the approximating polynomials should approximate the $F$-jumping numbers of $\fa$. This fact is indeed very clear in the case where $\fa$ is principal (c.f. Corollary \ref{principal-ideal-cor}). Here we begin working towards the proof of Theorem \ref{thm-roots-approx-fjn} which states, roughly, that to obtain a similar approximation property one should shift the index $e$ of the approximating polynomial by an integer we call a stable exponent. The definition as well as the subsequent lemma are inspired in \cite{Sato17}.
\begin{definition} \label{def-stable-exp}
	Given $\fa \sq R$ we call an integer $e_0>0$ a stable exponent (for $\fa$) if for all $n \in \N$ we have
	$$\tau(\fa^n) = \cC_R^{e_0} \cdot \fa^{np^{e_0}}.$$
\end{definition}
From the definition it follows that $e_0$ is stable if and only if for all $d>0$ we have  $\cC^{e_0} \fa^{n p^{e_0}} = \cC^{e_0 + d} \fa^{n p^{e_0 + d}}.$ Note that if $e_0$ is stable and $e > e_0$ then $e$ is also stable.  By Skoda's theorem (c.f. Proposition \ref{brsk-prop}), for $e_0$ to be stable it suffices to have the equality for integers $n$ with $0 \leq n < r$. In particular, stable exponents exist.

\begin{lemma} \label{stable-pe}
	Suppose $e_0$ is a stable exponent. Then for all $e>0$ we have
	$$\tau(\fa^{n/p^e}) = \cC_R^{e+e_0} \fa^{n p^{e_0}}$$
\end{lemma}

\begin{proof}
	For all $d > 0$ we have
	\begin{align*}
	\cC_R^{e + e_0} \fa^{n p^{e_0}} & = \cC_R^e \cC_R^{e_0} \fa^{n p^{e_0}} \\
	& = \cC_R^e \cC_R^{e_0 + d} \fa^{np^{e_0 + d}} \\
	& = \cC_R^{e + e_0 + d} \fa^{np^{e_0} + d}. \qedhere
	\end{align*}
\end{proof}

\begin{lemma} \label{N-to-testideal}
	Let $e_0$ be stable and fix $e > 0$ and $\gamma \in \F_p^e$. Then the following are equivalent.
	\begin{enumerate}[(1)]
		\item For all $\beta \in \F_p^{e_0}$, $N^{e_0 + e}_{(\beta, \gamma)} = 0 $.
		\item For all integers $s$ with $0 \leq s < r$,
		\begin{equation*}
		\tau \lp \fa^{s + \frac{|\gamma|}{p^e}} \rp = \tau \lp \fa^{s + \frac{|\gamma|+1}{p^e}} \rp. \label{test-ideal-eqn}
		\end{equation*}
		\item[(2')] The above holds for all nonnegative integers $s$.
		\item The set 
		$$\bigg(  \frac{|\gamma|}{p^e}, \frac{|\gamma| + 1}{p^e} \bigg] + \{0,1, \ds, r-1\}$$
		contains no F-jumping numbers of $\fa$.
		
		\item[(3')] The set
		$$\bigg(  \frac{|\gamma|}{p^e}, \frac{|\gamma| + 1}{p^e} \bigg] + \N_0$$
		contains no F-jumping numbers of $\fa$. 
	\end{enumerate}
\end{lemma}

\begin{proof}
	It is clear that (2) and (3) (resp. (2') and (3')) are equivalent. The equivalence between (2) and (2') (resp. (3) and (3')) follows from Skoda's theorem (c.f. Proposition \ref{brsk-prop}). We conclude that (2), (2'), (3) and (3') are all equivalent.
	
	Let us show that (1) and (2) are equivalent. From Theorem \ref{thm-summands-of-Ne} we see that  $N^{e_0 + e}_{(\beta, \gamma)} = 0$ for all $\beta$ if and only if
	$$\cC^e_R \cdot \fa^{|\gamma|p^{e_0} + s p^{e + e_0} + m} = \cC^e_R \cdot \fa^{|\gamma|p^{e_0} + s p^{e + e_0} + m + 1}.$$
	for all $0 \leq m < p^{e_0}$ and $0 \leq s < r$. 
	We conclude that (1) is true if and only if for all $s$ with $0 \leq s < r$ all containments in the chain of ideals
	$$\cC^e_R \cdot \fa^{|\gamma|p^{e_0} + s p^{e + e_0}} \supseteq \cC^e_R \cdot \fa^{|\gamma|p^{e_0} + s p^{e + e_0} + 1} \supseteq \cds \supseteq \cC^e_R \cdot \fa^{|\gamma|p^{e_0} + s p^{e + e_0} + p^{e_0} -1} \supseteq \cC^e_R \cdot \fa^{(|\gamma|+1)p^{e_0} + s p^{e + e_0}}$$
	are in fact equalities. This is equivalent to the statement that, for all $s$ with $0 \leq s < r$, the first and last ideals in the chain above are equal. But, by Lemma \ref{stable-pe}, the first ideal is $\tau(\fa^{s + |\gamma|/p^e})$ and the last ideal is $\tau(\fa^{s + (|\gamma| + 1)/p^e}).$ This completes the proof.
\end{proof}
We are now ready to state and prove the theorem.
\begin{theorem} \label{thm-roots-approx-fjn}
Let $e_0$ be a stable exponent and fix integers $e > 0$ and $0 \leq k < p^e$. Then the following are equivalent.
\begin{enumerate}[(1)]
	\item There is a root of $a^{e_0 + e}_\fa(s)$ in $[\frac{k}{p^e}, \frac{k + 1}{p^e})$.
	\item There is an $F$-jumping number of $\fa$ in $(\frac{k}{p^e}, \frac{k + 1}{p^e}] + \N$.
\end{enumerate}	
\end{theorem}
\begin{proof}
	Let $\gamma \in \F_p^e$ be the unique vector with $|\gamma| = k$. Let us show that (1) implies (2). Suppose that $n/p^{e_0 + e} \in [k/p^e, (k+1)/p^e)$ is a root of $b_\fa^{e_0 + e}(s)$. As $0 \leq n < p^{e_0 + e}$ there exists a unique vector $\alpha \in \F_p^{e_0 + e}$ with $|\alpha| = n$. Since $k p^{e_0} \leq n < (k+1)p^{e_0}$ it follows that $\alpha = (\beta, \gamma)$ for some $\beta \in \F_p^{e_0}$. By Lemma \ref{N-to-testideal} we get (2).
	
	Let us now show that (2) implies (1). By Lemma \ref{N-to-testideal}, (2) implies that $N^{e_0 + e}_\alpha \neq 0$ for some $\alpha = (\beta, \gamma)$. It follows that $|\alpha|/p^{e_0 + e}$ is a root of $a^{e_0 + e}_\fa(s)$. Since $|\alpha| = |\beta| + p^{e_0} k$ and $0 \leq |\beta| < p^{e_0}$ the root $|\alpha|/p^{e_0 + e}$ is in the required interval. 
\end{proof}

\section{The algebra $\cA_p$ and its modules} \label{scn-Ap-and-modules}
Having generalized the work of \Mustata \ to the case of a general ideal we now turn to generalizing the work of Bitoun given in \cite{Bitoun2018}. Our goal will be therefore to study the module
$$N_\fa := \frac{V^0 D_{R[t]} \cdot \delta}{V^1 D_{R[t]} \cdot \delta} = \lim_{\to e} N^e_\fa,$$
and to use it to detect the $F$-jumping numbers of $\fa$. 

When detecting the $\nu$-invariants of $\fa$ on the modules $N^e_\fa$ we made use of the fact that $N^e_\fa$ carries a module structure over the algebra $\F_p[s_{p^0}, \ds, s_{p^{e-1}}]$, where the $s_{p^i}$ satisfy the relation $s_{p^i}^p = s_{p^i}$. Similarly, $N_\fa$ carries a module structure over an algebra we call $\cA_p$. The algebra $\cA_p$ will be isomorphic to the algebra $\F_p[s_{p^0}, s_{p^1}, \ds]$ where we use all of the operators $s_{p^i}$,  but it will be convenient to define it in a different way. In this section we define $\cA_p$ and study its modules in an abstract setting before turning to our analysis of $N_\fa$. Our main tools for Section \ref{scn-A-module-N} will be Proposition \ref{bddimpliesdiscrete} and Proposition \ref{M-alpha-is-lim}.

Given a $p$-adic number $\alpha \in \Z_p$ we denote by $\alpha_i$ the integers with $0 \leq \alpha_i < p$ such that $\alpha = \sum_i p^i \alpha_i$. The fact that such a representation is unique establishes a bijection between $\Z_p$ and $\F_p^\N$. We will often identify the $p$-adic number $\alpha$ with its corresponding vector $(\alpha_0, \alpha_1, \ds) \in \F_p^\N$. Given $\alpha \in \Z_p$ we denote by $\alpha_{<e}$ the vector $\alpha_{<e}:=(\alpha_0, \ds, \alpha_{e-1})$. 

\begin{definition}
	We define the algebra $\cA_p$ by
	$$\cA_p := \frac{\F_p[\pi_0, \pi_1, \ds]}{(\pi_i^p - \pi_i : i \in \N_0)}$$
	and, given an integer $e > 0$, we define
	$$\cA^e_p := \frac{\F_p[\pi_0, \ds, \pi_{e-1}]}{(\pi_i^p - \pi_i : i = 0, \ds e-1)}.$$
\end{definition}
We will think of $\cA^e_p$ as subalgebras of $\cA_p$. In this way, one has an increasing union $\cA^1_p \sq \cA^2_p \sq \cds$ with $\cA_p = \bigcup_{e = 1}^\infty \cA^e_p$. In this next lemma we see that these algebras have already appeared in our previous discussion.
\begin{lemma} For $i \in \N_0$ let $s_{p^i}$ be the differential operator on $R[t_1, \ds, t_r]$ defined in Section \ref{scn-multi-eigen-dec}. Then we have:
	\begin{enumerate}[(a)]
		\item For each $e > 0$ the map $\cA^e_p \to \F_p[s_{p^0}, \ds, s_{p^{e-1}}]$ that sends $\pi_i$ to $s_{p^i}$ is an isomorphism.
		
		\item The map $\cA_p \to \F_p[s_{p^0}, s_{p^1}, \ds ]$ that sends $\pi_i$ to $s_{p^i}$ is an isomorphism.
	\end{enumerate}
\end{lemma}

\begin{proof}
	Let $\Fun(\F_p^e, \F_p)$ be the ring of $\F_p$-valued functions on $\F_p^e$ and let $x_0, \ds, x_{e-1}$ be the coordinate functions. As $\F_p^e$ is finite all functions $\F_p^e \to \F_p$ are polynomials in the $x_i$ and therefore the map $\cA^e_p \to \Fun(\F_p^e, \F_p)$ that sends $\pi_i$ to $x_i$ is surjective. As both algebras have the same number of elements (or the same dimension over $\F_p$) this map is an isomorphism.
	
	The composition $\Fun(\F_p^e, \F_p) \xrightarrow{\sim} \cA^e_p \to \F_p[s_{p^0}, \ds, s_{p^{e-1}}]$ sends $x_i$ to $s_{p^i}$. We construct an inverse to this composition. Consider the map $\phi: \F_p[s_{p^0}, \ds, s_{p^{e-1}}] \to \Fun(\F_p^e, \F_p)$ given by
	$$\phi(s_{p^i})(\alpha) := \frac{s_{p^i} \cdot t_1^{p^e - 1 - |\alpha|}}{t_1^{p^e - 1 - |\alpha|}},$$
	where $|\alpha| = \alpha_0 + p \alpha_1 + \cds p^{e-1} \alpha_{e-1}$ after identifying each $\alpha_i$ with its unique representative in $\{0, 1, \ds, p-1\}$. It follows from Proposition \ref{spi-combined-prop} (d) that this is indeed an $\F_p$-scalar and that, moreover, $\phi(s_{p^i})(\alpha) = \alpha_i$, i.e. $\phi(s_{p^i}) = x_i$. This completes the proof of (a).
	
	The statement in (b) follows from (a) and the fact that $\cA_p = \cup_e \cA_p^e$.
\end{proof}

As we have already mentioned in Proposition \ref{spi-combined-prop}(g), the algebras $\cA^e_p$ have a nice representation theory.

\begin{lemma} \label{eigenspace-lemma-Ae-modules}
	Let $e >0$ and let $M$ be an $\cA_p^e$-module. Then $M$ is spanned by its multi-eigenspaces for the operators $\pi_0, \ds, \pi_{e -1}$. That is,
	$$M = \bigoplus_{\alpha \in \F_p^e} M_\alpha$$
	where $\pi_i$ acts on $M_\alpha$ by the scalar $\alpha_i$.
\end{lemma}
\begin{proof}
	This is analogous to Proposition \ref{spi-combined-prop} (g).
\end{proof}
Suppose that $M$ is an $\cA_p$-module. Then it is also a $\cA_p^e$-module for all $e$, by restriction of scalars. It follows that for all $e$ one has a decomposition as above and these decompositions are compatible in the sense that $M_\alpha \sq M_{(\alpha, j)}$ for all $\alpha \in \F_p^e$ and $j \in \F_p$. Moreover, given $\alpha \in \Z_p$ one can define $M_\alpha := \{u \in M : \pi_i \cdot u = \alpha_i u \}$ and one gets an injection $\oplus_\alpha M_\alpha \to M$. But this need not be surjective -- indeed, one could have $M_\alpha = 0$ for all $\alpha \in \Z_p$. With this in mind we introduce the following definition.

\begin{definition} \label{dfn-root-etc}
	An $\cA_p$-module $M$ is discrete if the map
	$$\bigoplus_{\alpha \in \Z_p} M_\alpha \to M$$
	is an isomorphism and there are only finitely many $\alpha$ for which $M_\alpha \neq 0$. If $M$ is discrete a $p$-adic number $\alpha$ is a root of $M$ if $M_\alpha \neq 0$. An element $u \in M$ is said to be pure if there exists some $\alpha \in \Z_p$ such that $u \in M_\alpha$. This is equivalent to saying that for all $i \geq 0$ there exists some $\alpha_i \in \F_p$ such that $\pi_i \cdot u = \alpha_i u$. If $u \neq 0$ we say that $\alpha$ is the root of $u$. If $M$ is discrete with roots $\alpha_1, \ds, \alpha_n$ we say an integer $e>0$ is separating (for $M$) if $(\alpha_i)_{<e} \neq (\alpha_j)_{<e}$ for all $i \neq j$. It is clear that separating integers exist.
\end{definition}
A simple way in which $\cA_p$-modules may arise is via a limit procedure. Our module of interest $N_\fa$ arises this way, and this extra structure will be very useful in our analysis. We make the following definitions to address this situation.
\begin{definition}
	A filtered $\cA_p$-module consists of the following data:
	\begin{enumerate}[(i)]
		\item For each $e > 0$, a $\cA_p^e$-module $M^e$.
		\item For each $e > 0$, an $\cA_p^e$-module homomorphism $\phi_e: M^e \to M^{e+1}$. 
	\end{enumerate}
	These form an abelian category $\text{Filt $\cA_p$-mod}$ in a natural way. There is a functor $\lim: \text{Filt $\cA_p$-mod} \to \text{$\cA_p$-mod}$ that sends $(M^e, \phi_e)$ to $\lim_e M^e$.
\end{definition}

We will omit $\phi_e$ from the notation and simply refer to a filtered $\cA_p$-module by the collection of modules $(M^e)$. 

In order to study $\lim M^e$ it will be convenient to replace $M^e$ with another filtered $\cA_p$-module for which the maps $M^e \to M^{e+1}$ are injective.

\begin{lemma} \label{unionization}
	Let $(M^e)$ be a filtered $\cA_p$-module. For all $e \geq 0$ let $K^e$ be the set of all $u \in M^e$ for which there exists some $d > e$ such that $u$ maps to zero under the map $M^e \to M^d$. Let $\t{M}^e := M^e/K^e$. Then:
	\begin{enumerate}[(a)]
		\item The modules $\t{M}^e$ have the structure of a filtered $\cA_p$-module.
		\item There is a surjective map $(M^e) \to (\t{M}^e)$ of filtered $\cA_p$-modules.
		\item The maps $\t{M}^e \to \t{M}^{e + 1}$ are injective.
		\item The natural map $\lim M^e \to \lim \t{M}^e$ is an isomorphism of $\cA_p$-modules.
	\end{enumerate}  
\end{lemma}

\begin{proof}
	Omitted.
\end{proof}

\begin{definition}
	A filtered $\cA_p$-module $(M^e)$ is bounded if there exists some integer $K$ such that for all $e \geq 0$ one has
	$$\#\{\alpha \in \F_p^e | M^e_\alpha \neq 0\} \leq K,$$
	where the $M^e_\alpha$ are the eigenspaces given by Lemma \ref{eigenspace-lemma-Ae-modules}. 
\end{definition}

\begin{proposition} \label{bddimpliesdiscrete}
	If $(M^e)$ is a bounded filtered $\cA_p$-module then $\lim(M^e)$ is a discrete $\cA_p$-module.
\end{proposition}

\begin{proof}
	A quotient of a bounded filtered $\cA_p$-module is still bounded. By Lemma \ref{unionization} we may assume that the maps $\phi_e$ are injective and thus $M:= \lim M^e = \cup_e M^e$. In this case the sequence $\#\{\alpha \in \F_p^e | M^e_\alpha \neq 0\}$ is increasing and bounded, and therefore eventually constant. We conclude that there exists some $K'$ and $e_0$ such that for all $e \geq e_0$ one has $\#\{\alpha \in \F_p^e : M^e_\alpha \neq 0\} = K'$. 
	
	We now claim there are at most $K'$ $p$-adic integers $\alpha$ such that $M_\alpha \neq 0$. Indeed, suppose that $\beta_1, \ds, \beta_{K' + 1}$ are distinct $p$-adic integers such that $M_{\beta_i} \neq 0$ for all $i = 1, \ds, K'+1$. We may then pick $e \geq e_0$ large enough so that for all $i$ we have $M^e \cap M_{\beta_i} \neq 0$ and such that $(\beta_i)_{<e} \neq (\beta_j)_{<e}$ for all $i \neq j$. With this assumption we have $\# \{\alpha \in \F_p^e : M^e_\alpha \neq 0\} \geq K' + 1$, a contradiction.
	
	It remains to show that pure elements span the module $M$. For this suppose $v \in M$. Thus $v \in M^e$ for some $e \geq e_0$. Since $M^e = \oplus_{\alpha \in \F_p^e} M^e_\alpha$ it suffices to show that all elements of $M^e_\alpha$ are pure. For this purpose let $d > e \geq e_0$ and suppose that $\alpha \in \F_p^e$ is such that $M^e_\alpha \neq 0$. Since $M^e_\alpha = (\bigoplus_{\gamma \in \F_p^{d-e}} M^d_{(\alpha, \gamma)}) \cap M^e$ there exists a unique $\gamma \in \F_p^{d-e}$ such that $M^e_\alpha \sq M^d_{(\alpha, \gamma)}$ (otherwise the number of nonzero multi-eigenspaces increases past $K'$). It follows that all elements of $M^e_\alpha$ are pure as required.	
\end{proof}

\begin{lemma} \label{elargeimpliespure}
	Let $M$ be a discrete $\cA_p$-module and $e > 0$ be separating. Suppose that $u \in M$ is such that for all $i = 0,1, \ds, e-1$ there exists some $\alpha_i \in \F_p$ with $\pi_i \cdot u = \alpha_i \cdot u$. Then $u$ is pure.
\end{lemma}
\begin{proof}
	Because $M$ is discrete we may write $u = u_1 + \cds + u_n$ where $u_j$ is pure with root $\alpha_j$, where $\alpha_j \neq \alpha_k$ for $j \neq k$. But since $e$ is separating this is also the unique decomposition of $u$ into eigenvectors for $\{\pi_i : i =0, \ds, e-1\}$ coming from Lemma \ref{eigenspace-lemma-Ae-modules}. Since $u$ is itself an eigenvector for these operators all but one of the $u_i$ are zero and thus $u$ is pure.
\end{proof}

Let $(M^e)$ be a filtered $\cA_p$-module such that $\lim M^e$ is discrete (e.g. if $M^e$ is bounded) and let $\alpha \in \Z_p$. We define the modules
$$M^e_{\neq \alpha} := \bigoplus_{\beta \neq \alpha_{<e}} M^e_\beta.$$
Observe that the maps $M^e \to M^{e+1}$ restrict to maps $M^e_{\neq \alpha} \to M^{e+1}_{\neq \alpha}$ thus the collection $(M^e_{\neq \alpha})$ acquires a filtered $\cA_p$-module structure. It is in fact a subobject of $(M^e)$ in the category of filtered $\cA_p$-modules.

We can also consider
$$M^e_\alpha := M^e_{\alpha_{<e}}.$$
In this case, even though $M^e_\alpha \sq M^e$, the maps $M^e \to M^{e+1}$ do not restrict to $M^e_{\alpha} \to M^{e+1}_\alpha$ but we still have natural maps $M^e_\alpha \to M^{e+1}_\alpha$ that send an element $u \in M^e_{(\alpha_0, \ds, \alpha_{e-1})}$ to the $(\alpha_0, \ds, \alpha_e)$-component of its image under $M^e \to M^{e+1}$. This makes $(M^e)_\alpha$ into a filtered $\cA_p$-module for which $\lim M^e_\alpha$ is discrete with at most a single root $\alpha$. In the following lemma we observe that $(M^e_\alpha)$ is in fact a quotient of $(M^e)$ in the category of filtered $\cA_p$-modules.

\begin{lemma} \label{filtAmod-exseq} \ 
	\begin{enumerate}
		\item One has a short exact sequence
		$$0 \to M^e_{\neq \alpha} \to M^e \to M^e_\alpha \to 0$$
		of filtered $\cA_p$-modules.
		\item If $\lim M^e$ is discrete with roots $\alpha_1, \ds, \alpha_n$ there is a sequence
		$$0 \to \bigcap_{i = 1}^n M^e_{\neq \alpha_i} \to M^e \to \bigoplus_{i = 1}^n M^e_{\alpha_i} \to 0$$
		that is exact whenever $e$ is separating.
	\end{enumerate}
\end{lemma}

\begin{proof}
	The claim in (i) follows directly from the descriptions. The existence and left exactness of the sequence in (ii) is clear. The surjectivity of $M^e \to \oplus M^e_{\alpha_i}$ follows from the separating hypothesis.
\end{proof}

\begin{proposition} \label{M-alpha-is-lim}
	Let $(M^e)$ be a filtered $\cA_p$-module such that $M:=\lim M^e$ is discrete. Then the natural map
	$$M_\beta \to \lim M^e_\beta$$
	is an isomorphism for all $\beta \in \Z_p$.
\end{proposition}

\begin{proof}
	Consider first the case where $\beta$ is not a root -- that is, where $M_\beta = 0$. We must show that $\lim M^e_\beta = 0$. Let $e$ be separating and, by enlarging it if necessary, assume that $\beta_{<e} \neq \alpha_{<e}$ for all roots $\alpha$. Let $0 \neq u \in M^e_\beta = M^e_{(\beta_0, \ds, \beta_{e-1})}$. By Lemma \ref{elargeimpliespure} the image of $u$ in $M$ is pure and, if nonzero, would have a root $\alpha$ with $\beta_{<e} = \alpha_{<e}$, a contradiction to our assumption on $e$. We conclude that the image of $u$ in $M$ is zero and thus there is a $d > e$ such that the image of $u$ under $M^e \to M^d$ is zero. The commutativity of the diagram
	$$\begin{tikzcd}
	M^e \arrow[r] \arrow[d] & M^e_\beta \arrow[d] \\
	M^d \arrow[r] & M^d_\beta
	\end{tikzcd}$$
	implies that the image of $u$ is also zero under $M^e_\beta \to M^d_\beta$.
	
	It remains to prove the statement for the roots $\alpha_1, \ds, \alpha_n$. We will prove that the map $\psi: M = \oplus_i M_{\alpha_i} \to \oplus_i \lim M^e_{\alpha_i}$ is an isomorphism and this will give the result.
	
	From Lemma \ref{filtAmod-exseq} (ii) and the exactness of $\lim$ we have a short exact sequence
	$$0 \to \lim \bigcap_{i=1}^n M^e_{\neq \alpha_i} \to M \to \lim \bigoplus_{i = 1}^n M^e_{\alpha_i} \to 0.$$
	There is a natural isomorphism $\lim \oplus_i M^e_{\alpha_i} \cong \oplus_i \lim M^e_{\alpha_i}$ and the composition
	$$M \to \lim \bigoplus_{i = 1}^n M^e_{\alpha_i} \xrightarrow{\sim} \bigoplus_{i = 1}^n \lim M^e_{\alpha_i}$$
	is equal to $\psi$. Therefore it suffices to show that $\lim \cap_i M^e_{\neq \alpha_i} = 0$. But this follows because
	$$\bigcap_{i = 1}^n M^e_{\neq \alpha_i} = \bigoplus_{\{\beta \in \F_p^e : \beta \neq (\alpha_i)_{<e} \forall i \}} M^e_\beta$$
	and if $e$ is separating every nonzero element of $M^e_\beta$ has a pure image in $M$. Therefore if $\beta \neq (\alpha_i)_{<e}$ for all $i$ one has that the image of $M^e_\beta$ in $M$ is zero -- thus also its image in $\lim \cap_i M^e_{\neq \alpha_i}.$ 
\end{proof}

\pagebreak

\section{Bernstein-Sato roots} \label{scn-A-module-N}

Let $N^e := N^e_\fa$ be the modules defined in Section \ref{scn-multi-eigen-dec}. Recall that these depend on a choice of generators $\fa = (f_1, \ds, f_r)$ for the ideal $\fa$. 

As seen in Section \ref{scn-multi-eigen-dec} the module $N^e$ has an $\cA_p^e$-module structure, where $\pi_i$ acts via the operator $s_{p^i}$. Moreover, we have maps $N^e \to N^{e+1}$ that are linear over $\cA_p^e$. This gives the collection $(N^e)$ an filtered $\cA_p$-module structure. We let $N := N_\fa = \lim N^e$. 

Throughout this section we will make use of the fact that the collection of $F$-jumping numbers for $(R, \fa)$ is discrete and rational. In the case where $R$ is of finite type over a field this was shown in \cite{BMSm2008}, and we refer to \cite{SchTucTI} for the more general statement.

\begin{proposition} \label{N-is-discrete}
	The $\cA_p$-module $N_\fa$ is discrete.
\end{proposition}
\begin{proof}
	By Proposition \ref{bddimpliesdiscrete} it suffices to show that the filtered $\cA_p$-module $(N^e)$ is bounded. In turn, for this it suffices to show that there exists some $K > 0$ such that $\#\{\alpha \in \F_p^e : N^e_\alpha \neq 0\} \leq K$ for $e$ large enough. 
	
	Let $e_0$ be a stable exponent and, by making use of the discreteness of $F$-jumping numbers, let $e \gg 0$ be large enough so that every interval of the form $(\frac{k}{p^e}, \frac{k+1}{p^e}]$ with $k \in \N_0$ contains at most one $F$-jumping number of $\fa$. 
	
	Recall from Lemma \ref{N-to-testideal} that $N^{e_0 + e}_{(\beta, \gamma)} \neq 0$ if and only if there is an $F$-jumping number of $\fa$ in the set $(\frac{|\gamma|}{p^e}, \frac{|\gamma| + 1}{p^e}] + \{0, \ds, r-1\}$. If we let $N$ be the number of $F$-jumping numbers contained in $(0,r]$ it then follows that $\#\{\alpha \in \F_p^{e_0 + e} : N^{e_0 + e}_\alpha \neq 0\} \leq p^{e_0} N$. 
\end{proof}
It follows from the definition of discreteness that $N_\fa$ has a finite set of roots (c.f. Definition \ref{dfn-root-etc}), to which we give a special name.
\begin{definition} \label{dfn-BSroot}
	The Bernstein-Sato roots of $\fa$ are the roots of the $\cA_p$-module $N_\fa$. That is to say, if $\alpha \in \Z_p$ has $p$-adic expansion $\alpha_0 + p\alpha_1 + \cds$ (i.e. $\alpha_i \in \{0, 1, \ds, p-1\})$ then $\alpha$ is a Bernstein-Sato root if and only if the multi-eigenspace
	$$(N_\fa)_\alpha := \{u \in N : s_{p^i} \cdot u = \alpha_i u \text{ for all } i\}$$
	is nonzero.  We denote the set of all Bernstein-Sato roots of $\fa$ by $BS(\fa)$ which, a priori, is just a set of $p$-adic integers.
\end{definition}

\subsection{First remarks} \label{subscn-first-rem}

We would like to first address why the above definition follows naturally from the one in characteristic zero.  While issues regarding $p$-torsion in the corresponding $D$-modules have not allowed us to turn these observations into reasonable statements regarding reduction modulo $p$ of Bernstein-Sato roots, we hope that they help motivate Definition \ref{dfn-BSroot}. 

Recall that, over $\C$, the Bernstein-Sato polynomial of $\fa$ is defined as the minimal polynomial for the action of $s_1$ on $N_\fa$. The existence of a Bernstein-Sato polynomial for $\fa$ with roots $\alpha^{(1)}, \ds, \alpha^{(t)}$ is therefore equivalent to the existence of a decomposition $N_\fa = (N_\fa)_{\alpha^{(1)}} \oplus \cds \oplus (N_\fa)_{\alpha^{(t)}}$ where $(N_\fa)_{\alpha^{(j)}}$ is the $\alpha^{(j)}$-generalized eigenspace for the action of $s_1$ on $N_\fa$, i.e. we can find some $M > 0$ such that $(N_\fa)_{\alpha^{(j)}} := \{u \in N : (s_1 - \alpha^{(j)})^M \cdot u = 0\}.$ From the recursive properties of the operators $s_m$ (c.f. Proposition \ref{prop-s_m-properties}), we see that $(s_1 - \alpha)^M \cdot u = 0$ is equivalent to $(s_m - {\alpha \choose m})^M \cdot u = 0$ for all $m$, where ${x \choose m} := x (x-1) \cds (x - m + 1)/m!$. 

If $\alpha$ is in $\Z_{(p)}$ then ${\alpha \choose m}$ is also in $\Z_{(p)}$: this is clear when $\alpha$ is an integer, so we approximate $\alpha$ by integers in $p$-adic norm and observe that ${x \choose m}$, being a polynomial, is $p$-adically continuous (see \cite{KCon}). 

It follows that, in characteristic $p$, the equations $(s_{p^i} - {\alpha \choose p^i})^M \cdot u = 0$ still make sense as long as $\alpha \in \Z_{(p)}$ and, since $M$ can be replaced by a power of $p$, they are equivalent to $(s_{p^i} - {\alpha \choose p^i}) \cdot u = 0$ (c.f. Proposition \ref{prop-s_m-properties} (e)). Finally, we observe that ${\alpha \choose p^i} \equiv \alpha_i \mod p$ (we again approximate $\alpha$ by integers, for which the statement is just Lucas' theorem, Lemma \ref{lucas-thm-lemma}), and we arrive at Definition \ref{dfn-BSroot}. 

In this section we will see that Berstein-Sato roots satisfy other nice properties beyond this mild compatibility with the definition from characteristic zero. To begin with, the set of Bernstein-Sato roots is independent of the initial choice of generators for $\fa$ (Corollary \ref{cor-bs-roots-indep-gens}) and, in analogy with the characteristic zero case, Bernstein-Sato roots are negative and rational (Theorem \ref{BSroots-are-rational}). We will also show that they encode some information about the $F$-jumping numbers of $\fa$ (Theorem \ref{thm-bsroots-and-fjn}). Moreover, they give the correct notion for monomial ideals \cite{QG19b}. After providing another characterization of Bernstein-Sato roots (Proposition \ref{prop-BSroots-newChar}) we give a few examples.

\subsection{Preliminary results}

We begin with a few results that will be crucial in understanding the vanishing of the eigenspaces $N^e_\alpha$ as $e$ goes to infinity.

\begin{lemma} \label{zero_persists_hard} \label{zero_remains_zero}
	Fix integers $0<e<d$, $\alpha \in \F_p^e$ and $\beta \in \F_p^{d - e}$. Then the image of $N^e_\alpha$ in $N^d_{(\alpha, \beta)}$ is zero if and only if $N^d_{(\alpha, \beta)} = 0$. If particular, if $N^e_\alpha = 0$. Then $N^{e+1}_{(\alpha,j)} = 0$ for all $0 \leq j < p$.
\end{lemma}
\begin{proof}
	The $(\Leftarrow)$ implication is clear. For $(\Rightarrow)$ let $n = |\alpha|$, $k = |\beta|$ and $m = |(\alpha, \beta)|$. Thus $m = n + p^e k$. By Theorem \ref{thm-summands-of-Ne} it is enough to show that for all $0 \leq t < r$ we have
	$$\fa^{m + tp^d} \sq D^d \cdot \fa^{m + t p^d + 1}.$$
	Let $\ul{a} \in \N_0^r$ be such that $|\ul{a}| = m + tp^d$ and we will show that $f^{\ul{a}} \in D^d \cdot \fa^{m + tp^d + 1}$. 
	
	Let $\ul{a_0}$ and $\ul{a_1}$ be the unique $r$-tuples of nonnegative integers with $\ul{a} = \ul{a_0} + p^e \ul{a_1}$ and $\ul{a_0} \in \{0, \ds, p^e -1\}^r$. Let $n' := |\ul{a_0}|$. Since by assumption the map
	$$\frac{D^e \cdot \fa^{n'}}{D^e \cdot \fa^{n' + 1}} \t{Q}^e_{\ul{a_0}} \xrightarrow{f^{\ul{a_1}p^e}} \frac{D^d \cdot \fa^{m + tp^d}}{D^d \cdot \fa^{m + tp^d + 1}} \t{Q}^d_{\ul{a}}$$
	is zero (c.f. Remark \ref{Qe-to-Qe+1-rmk}), the image of $f^{\ul{a_0}} \t{Q}^e_{\ul{a_0}}$ is zero. It follows that $f^{\ul{a}} \in D^d \cdot \fa^{m + tp^d + 1}$ as required.
\end{proof}
Recall from Section \ref{scn-Ap-and-modules} that if $\alpha \in \Z_p$ and $e> 0$ we denote $N^e_\alpha := N^e_{(\alpha_0, \ds, \alpha_{e-1})}$, and that these modules have a filtered $\cA_p$-module structure.
\begin{proposition} \label{when-alpha-is-root}
	Fix $\alpha \in \Z_p$ with $p$-adic expansion $\alpha = \alpha_0 + p \alpha_1 + p^2 \alpha_2 + \cds$. Then the following are equivalent.
	\begin{enumerate}[(1)]
		\item The $p$-adic number $\alpha$ is a Bernstein-Sato root of $\fa$.
		\item For all $e \geq 0$ we have $N^e_{(\alpha_0, \ds, \alpha_{e-1})} \neq 0$. 
		\item There is a sequence $e_n \nearrow \infty$ such that $N^{e_n}_{(\alpha_0, \alpha_1, \ds, \alpha_{e_n})} \neq 0$.
	\end{enumerate}
\end{proposition}
\begin{proof}
	As $N$ is discrete it follows from Proposition \ref{M-alpha-is-lim} that $N_\alpha \cong \lim (N^e_\alpha)$. The fact that (1) implies (2) then follows from Lemma \ref{zero_remains_zero}: if $N^e_\alpha = 0$ for some $e$ then $N^{e'}_\alpha = 0$ for all $e' > e$. Statement (2) implies (3) trivially. For (3) implies (1) suppose for a contradiction that $\alpha$ is not a root; that is, $\lim N^e_\alpha = 0$ (c.f. Proposition \ref{M-alpha-is-lim}). As $N^1$ is a finitely generated $R$-module (in fact, so are all of the $N^e$) and the maps $N^e \to N^{e+1}$ are $R$-module homomorphisms it follows that for all $d \gg 0$ we have $\im(N^1_\alpha \to N^d_\alpha) = 0$. From Lemma \ref{zero_remains_zero} we see that $N^d_\alpha = 0$ for all $d \gg 0$, which contradicts (3).
\end{proof}
\begin{corollary} \label{cor-bs-roots-indep-gens}
	The set of Bernstein-Sato roots of $\fa$ is independent of the initial choice of generators. 
\end{corollary}

\pagebreak
\begin{lemma} \label{root-predynamics}
	Let $e_0$ be a stable exponent, $e > 0$ and $\gamma \in \F_p^e$. Then the following are equivalent:
	\begin{enumerate}[(1)]
		\item For all $\beta \in \F_p^{e_0}$ we have $N^{e_0 + e}_{(\beta, \gamma)} = 0$.
		\item For all $\beta \in \F_p^{e_0}$ and $j \in \F_p$ we have $N^{e_0 + 1 + e}_{(\beta, j, \gamma)} = 0$. 
		\item For all $\beta' \in \F_p^{e_0 + 1}$ we have $N^{e_0 + 1 + e }_{(\beta', \gamma)} = 0$.
	\end{enumerate}
\end{lemma}
\begin{proof}
	It follows from Lemma \ref{N-to-testideal} that statement (1) is equivalent to the fact that 
	$$S_e := \bigg(\frac{|\gamma|}{p^e}, \frac{|\gamma|+1}{p^e}\bigg] + \N$$
	contains no F-jumping numbers of $\fa$. Another application of Lemma \ref{N-to-testideal} gives that statement (2) is equivalent to the fact that for all $0 \leq j < p$ the set
	$$S_{e+1, j} := \bigg(\frac{j + p|\gamma|}{p^{e+1}}, \frac{j + p|\gamma|+1}{p^{e+1}}\bigg] + \N$$
	contains no F-jumping numbers of $\fa$. As we have $S_e = \bigcup_{j = 0}^{p-1} S_{e+1,j}$ the equivalence between (1) and (2) follows. Finally, (2) is trivially equivalent to (3).
\end{proof}
\subsection{Rationality and negativity of the Bernstein-Sato roots}
In this subsection we prove the following statement. Recall we say a $p$-adic number is rational if it lies in $\Z_{(p)}$.
\begin{theorem} \label{BSroots-are-rational}
	The Bernstein-Sato roots of $\fa$ are negative rational numbers.
\end{theorem}
The analogous result in characteristic zero was proven by Kashiwara in \cite{Kas76} by using resolution of singularities. 

A posteriori (c.f. Theorem \ref{thm-bsroots-and-fjn}) we see that the rationality of Bernstein-Sato roots is analogous to the rationality of the F-jumping numbers. In \cite{BMSm2008} one sees that rationality of F-jumping numbers follows from their discreteness together with the fact that if $\lambda$ is an F-jumping number for $\fa$ then so is $p \lambda$. The following lemma is the analogue of this latter property for Bernstein-Sato roots.
\begin{lemma} \label{root-dynamics}
	Let $\alpha \in \Z_p$. If $p \alpha + \Z$ contains a Bernstein-Sato root then so does $\alpha + \Z$.
\end{lemma}
\begin{proof}
	Let us suppose that $\alpha + \Z$ contains no roots and we will show that $p \alpha + \Z$ contains no roots. Thus let $j \in \Z$ and let us show that $p \alpha + j$ is not a root. First we change $\alpha$ if necessary to assume that $0 \leq j < p$.
	
	Let $\beta \in \F_p^{e_0}$. By assumption we have that $|\beta| - \sum_{i = 0}^{e_0 - 1} \alpha_i p^i + \alpha$
	is not a root and, by Proposition \ref{when-alpha-is-root}, we can find $e>0$ large enough so that
	$$N^{e_0 + e}_{(\beta_0, \ds, \beta_{e_0 - 1}, \alpha_{e_0}, \ds, \alpha_{e_0 + e -1 })} = 0.$$ 
	As the set $\F_p^{e_0}$ is finite, this $e$ may be chosen independently of $\beta$. By Lemma \ref{root-predynamics} we have
	$$N^{e_0 + e+ 1}_{(\beta', \alpha_{e_0}, \ds, \alpha_{e_0 + e - 1})} = 0$$
	for all $\beta' \in \F_p^{e_0 + 1}$. This holds in particular for $\beta'=(j,\alpha_0, \ds, \alpha_{e_0 - 1})$ and again by Proposition \ref{when-alpha-is-root} we conclude that $j + p\alpha$ is not a root.
\end{proof}
We will also need the following result from \cite{BMSm2008}.
\begin{proposition}[{\cite[Prop. 2.14]{BMSm2008}}]
	Given $\fa \sq R$ as above, and $\lambda \in \R_{>0}$, there exists some $\epsilon > 0$ such that whenever $\lambda < r/p^e < \lambda + \epsilon$ we have $\tau(\fa^\lambda) = \cC^e_R \cdot \fa^r$.	
	
\end{proposition}
\begin{corollary} \label{cor-eiri-sequence}
	Let $(e_i)$ and $(r_i)$ be sequences of nonnegative integers such that $e_i/r_i \searrow \lambda$. Then for all $i \gg 0$ we have $\tau(\fa^{\lambda}) = \cC^{e_i}_R \cdot \fa^{r_i}$. 
\end{corollary}
We are now ready to prove the theorem. We will use results from the appendix (Section \ref{subsn-p-adic-exp}).
\begin{proof} [Proof of Theorem \ref{BSroots-are-rational}]
	Given $\alpha \in \Z_p$ we denote by $\alpha^{(n)}$ the new $p$-adic number $\alpha^{(n)} := \alpha_n + p \alpha_{n+1} + p^2 \alpha_{n+2} + \cds$. Observe that $\alpha \equiv p^n \alpha^{(n)} \mod \Z$ for all $\alpha \in \Z_p$. Let $\t{BS}(\fa)$ be the image of $BS(\fa)$ in $\Z_p/\Z$. That is, $\t{BS}(\fa) := \{\alpha + \Z : \alpha \in BS(\fa)\} \sq \Z_p/\Z.$
	
	From Lemma \ref{root-dynamics} it follows that $\t{BS}(\fa)$ is closed under the operation $[\alpha + \Z \mapsto \alpha^{(1)} + \Z]$. As $BS(\fa)$ is finite by Proposition \ref{N-is-discrete}, so is $\t{BS}(\fa)$.
	
	It follows that if $\alpha$ is a Bernstein-Sato root then there exist positive integers $n < m$ such that $\alpha^{(n)} + \Z = \alpha^{(m)} + \Z$. We then have
	$$\alpha \equiv p^m \alpha^{(m)} \equiv p^m \alpha^{(n)} \equiv p^{m-n} \alpha \text{ mod } \Z.$$
	
	Therefore there exists some $c \in \Z$ such that $\alpha = p^{m-n} \alpha + c$ and thus $\alpha = c/(p^{m-n} - 1) \in \Z_{(p)}$, which shows that $\alpha$ is rational.
	
	Let us now show that $\alpha$ is negative. As $\alpha \in \Z_{(p)}$ by we may write $\alpha = m + p^d \gamma$ where $m \in \{0,1, \ds, p^d -1\}$ and $\gamma \in \Z_{(p)}$ with $-1 \leq \gamma \leq 0$, i.e. $\gamma$ has a periodic expansion (c.f. Lemma \ref{periodic-iff-rational}). In fact, we may assume that $(p^d - 1)\alpha \in \Z$ and that $d$ is a stable exponent. Let $\lambda := - \gamma$, , and we want to show that $m < p^d \lambda$. 
	
	Suppose for a contradiction that $m \geq p^d \lambda$. Fix $s \in \{0, 1, \ds, r-1\}$ and let $K_{e,s} := \alpha_0 + p \alpha_1 + \cds + p^{(e+1)d-1}\alpha_{(e+1)d-1} + sp^{(e+1)d} $. By Lemma \ref{zp-to-padic-lemma}, we have $K_{e,s} = m + p^d \lambda(p^{ed -1}) + sp^{(e+1)d}$. We claim that for $e$ large enough we have $\cC^{(e+1)d}_R \cdot \fa^{K_{e,s}} = \cC^{(e+1)d}_R \cdot \fa^{K_{e,s} +1 }$. If we then take $e$ large enough so that this holds for all $s \in \{0, 1, \ds, r-1\}$ then, by Theorem \ref{thm-summands-of-Ne} we have $N^{(e+1)d}_{(\alpha_0, \ds, \alpha_{(e+1)d-1}) }= 0$ which, by Proposition \ref{when-alpha-is-root}, is a contradiction. 
	
	To prove the claim, consider the chain of ideals
	$$\cC^{(e+1)d}_R \cdot \fa^{K_{e,s}} \supseteq \cC^{(e+1)d}_R \cdot \fa^{K_{e,s} + 1} \supseteq \cds \supseteq \cC^{(e+1)d}_R \cdot \fa^{p^d\lambda(p^{ed} - 1)+p^d + sp^{(e+1)d}}$$
	(at each stage, add 1 to the exponent of $\fa$). 
	
	Since $d$ is a stable exponent, the ideal on the right is $\tau(\fa^{\frac{\lambda(p^{ed} -1) + 1}{p^{ed}}+s})$ (c.f. Lemma \ref{stable-pe}). Observe that the sequence 
	$$e \mapsto \frac{K_{e,s}}{p^{(e+1)d}} = \lambda + s + \frac{m - p^d \lambda}{p^{(e+1)d}}$$
	decreases to $\lambda + s$ by the assumption $m - p^d \lambda \geq 0$. By Corollary \ref{cor-eiri-sequence} we have, for $e$ large enough,
	$$\cC^{(e+1)d}_R \cdot \fa^{K_{e,s}} = \tau(\fa^\lambda)$$
	and thus, by making $e$ even larger if necessary, the ideals on the left and right of the chain coincide. In particular, $\cC^{(e+1)d}_R \cdot \fa^{K_{e,s}} = \cC^{(e+1)d}_R \cdot \fa^{K_{e,s} + 1}$ as required.		
\end{proof}
\subsection{Bernstein-Sato roots and F-jumping numbers}
Our next theorem shows that some information about the $F$-jumping numbers of $\fa$ is encoded in this set of Bernstein-Sato roots. Let us denote by $FJ(\fa)$ the set of $F$-jumping numbers of $\fa$, and recall that $BS(\fa)$ is the set of Bernstein-Sato roots of $\fa$. With this notation, the theorem is as follows.

\begin{theorem} \label{thm-bsroots-and-fjn}
	We have
	$$BS(\fa) + \Z = -FJ(\fa) \cap \Z_{(p)} + \Z.$$
\end{theorem}
We will make use of results from the appendix (Section \ref{subsn-p-adic-exp}) in the proof.
\begin{proof}
	Let $\lambda \in \Z_{(p)}$ with $0 < \lambda \leq 1$ be such that $\lambda + \Z$ contains an $F$-jumping number of $\fa$. We will show that $-\lambda \in BS(\fa) + \Z$ and this will prove the inclusion $(\supseteq)$. Let $d$ be large enough so that $(p^d - 1) \lambda \in \N$ and so that $d$ is a stable exponent and let $\gamma := - \lambda$. Observe that for all $e > 0$ we have
	$$\frac{\lambda(p^{ed} - 1)}{p^{ed}} < \lambda \leq \frac{\lambda (p^{ed} - 1) + 1}{p^{ed}}$$
	where $\lambda(p^{ed}-1) \in \N$. By Lemma \ref{zp-to-padic-lemma} we have $\lambda(p^{ed} - 1) = \gamma_0 + p \gamma_1 + \cds + p^{ed -1} \gamma_{ed-1} = |\gamma_{< ed}|$. From Lemma \ref{N-to-testideal} we conclude that for all $e$ there exists some $\beta \in \F_p^d$ such that $N^{d + ed}_{(\beta, \gamma_{<ed})} \neq 0.$ As $\F_p^d$ is a finite set there exists some $\beta_0$ and a sequence $e_n \nearrow \infty$ such that $N^{d + e_nd}_{(\beta_0, \gamma_{<e_n d})} \neq 0$. From Proposition \ref{when-alpha-is-root} it follows that $|\beta_0| - |\gamma_{d}| + \gamma$ is a root and thus $\gamma \in BS(\fa) + \Z$ as required.

	Suppose now that $\alpha \in BS(\fa)$. By Theorem \ref{BSroots-are-rational} we have that $\alpha \in \Z_{(p)}$. Choose $\gamma \in \Z_{(p)}$ as in Lemma \ref{p-adic-lemma3} -- that is, $-1 \leq \gamma \leq 0$ has a periodic expansion that is eventually equal to the expansion of $\alpha$, and in particular $\gamma + \Z = \alpha + \Z$. Let $\lambda := -\gamma$, and we will show that $\lambda + s$ is an $F$-jumping number for some $s \in \N$, which will complete the proof.
	
	Let $d$ be the length of the period of $\gamma$ and let $\t{\gamma} = (\gamma_0, \ds, \gamma_{d-1})$ be the period. We may assume $d$ is large enough to be a stable exponent and so that $\gamma_i = \alpha_i$ for all $i\geq d$. As $\alpha$ is a root it follows from Proposition \ref{when-alpha-is-root} that for all $e > 0$ we have $N^{(e+1)d}_\alpha \neq 0$. The first $(e+1)d$ digits of the expansion of $\alpha$ are
	$$(\alpha_0, \ds, \alpha_{d-1}, \t{\gamma}, \ds , \t{\gamma})$$
	where $\t{\gamma}$ is repeated $e$ times. By Lemma \ref{zp-to-padic-lemma} we have $|(\t{\gamma}, \ds, \t{\gamma})| = \lambda(p^{ed} -1)$ and, by Lemma \ref{N-to-testideal}, the set
	$$\bigg( \frac{\lambda(p^{ed} -1)}{p^{ed}}, \frac{\lambda(p^{ed} -1) + 1}{p^{ed}} \bigg] + \{0, \ds, r-1\}$$
	contains an F-jumping number of $\fa$ for all $e > 0$. As $\{0, \ds, r-1\}$ is finite there exists some $0 \leq s < r$ and a subsequence $e_n \nearrow \infty$ such that
	$$\bigg( \frac{\lambda(p^{e_nd} -1)}{p^{e_nd}}, \frac{\lambda(p^{e_nd} -1) + 1}{p^{e_nd}} \bigg] + s$$
	contains an F-jumping number. It follows that $\lambda + s$ is an F-jumping number as required.
\end{proof}

\subsection{Another characterizations of Bernstein-Sato roots}

Here we provide one more characterization of the Bernstein-Sato roots of an ideal $\fa$ that simplifies a lot of computations. As usual, let $R$ be regular and $F$-finite and let $\fa \sq R$ be an ideal. We denote by $N^e$ the modules $N^e_\fa$ constructed in Section \ref{scn-multi-eigen-dec} (after a choice of generators $(f_1, \ds, f_r)$ for $\fa$). Recall that, given a positive integer $e$, we denote by $\nu^\bullet_\fa(p^e)$ the set of $\nu$-invariants of level $e$ for $\fa$ (c.f. Definition \ref{def-nu-invt}, Proposition \ref{nu-invt-trun-ti-prop}).

\begin{lemma}
	The $\nu$-invariants for $\fa$ come in a decreasing chain
	$$\nu^\bullet_\fa(p^0) \supseteq \nu^\bullet_\fa(p^1) \supseteq \cds \supseteq \nu^\bullet_\fa(p^e) \supseteq \nu^\bullet_\fa(p^{e+1}) \supseteq \cds$$
\end{lemma}
\begin{proof}
	Given an ideal $J$, we have $J^{[p^{e+1}]} = (J^{[p^e]})^{[p]}$. It then follows from the definition that $\nu^{J}_\fa(p^{e+1}) = \nu^{J^{[p]}}_\fa(p^e)$. 	
\end{proof}
We can now state our new characterization.
\pagebreak
\begin{proposition} \label{prop-BSroots-newChar}
		The following sets are equal.
		\begin{enumerate}[(a)]
			\item The set of Bernstein-Sato roots of the ideal $\fa$.
			\item The set of $p$-adic limits of sequences $(\nu_e) \sq \N$ where $\nu_e \in \nu^\bullet_\fa(p^e)$. 
			\item The set
			$$\bigcap_{e = 0}^\infty \overline{\nu^\bullet_\fa(p^e)},$$
			where $(\bar{ \ } )$ stands for $p$-adic closure.
		\end{enumerate}
\end{proposition}
\begin{proof}
	The equality between (b) and (c) is a general fact about metric spaces, which we prove in Lemma \ref{lemma-met-spc}. We therefore prove that (a) and (b) are equal. 
	
	Suppose that $\alpha$ is a Bernstein-Sato root of $\fa$ and let $\alpha = \alpha_0 + p \alpha_1 + \cds$ be its $p$-adic expansion. By Proposition \ref{when-alpha-is-root}, for all $e > 0$ we have $N^e_{(\alpha_0, \alpha_1, \ds, \alpha_{e - 1})}$. From Theorem \ref{thm-summands-of-Ne} (b) it follows that for all $e$ there exists some $s_e \in \{0, 1, \ds, r-1\}$ such that, if $\nu_e = \alpha_0 + p \alpha_1 + \cds + p^{e-1} \alpha_{e-1} + p^e s_e$, then $\cC^e_R \cdot \fa^{\nu_e} \neq \cC^e_R \cdot \fa^{\nu_e + 1}$, that is $\nu_e \in \nu^\bullet_\fa(p^e)$ (c.f. Proposition \ref{nu-invt-trun-ti-prop}). Since $\alpha$ is the $p$-adic limit of the sequence $(\nu_e)$, it follows that $\alpha$ is in (b).
	
	Now suppose that $\alpha$ is the $p$-adic limit of a sequence $(\nu_e)$ with $\nu_e \in \nu^\bullet_\fa(p^e)$; that is, $\cC^e_R \cdot \fa^{\nu_e} \neq \cC^e_R \cdot \fa^{\nu_e + 1}$. By Corollary \ref{nu-invt-dynamics-cor} we may assume that $\nu_e \leq r p^e$ for all $e$. By Lemma \ref{nu-to-alpha-lemma} we have $\nu_e - \lfloor \nu_e/p^e \rfloor p^e = \alpha_0 + p \alpha_1 + \cds + p^{e - 1} \alpha_{e -1}$ and, by Theorem \ref{thm-summands-of-Ne} (b), we have $N^e_{(\alpha_0, \ds, \alpha_{e -1})} \neq 0$ for all $e$. It follows from Proposition \ref{when-alpha-is-root} that $\alpha$ is a Bernstein-Sato root of $\fa$. 
\end{proof}
\begin{lemma} \label{lemma-met-spc}
	Let $X$ be a metric space and $X_0 \supseteq X_1 \supseteq X_2 \supseteq \cds $ be a decreasing sequence of subspaces of $X$. Let $Y$ be the subspace of $X$ given by $Y := \{\lim_{i \to \infty} (x_e) : x_e \in X_e\}$. Then
	$$Y = \bigcap_{e = 0}^\infty \overline{X_e},$$
	where $\overline{X_e}$ is the closure of $X_e$ in $X$. 
\end{lemma}
\begin{proof}
	If $y \in Y$ we may write $y = \lim_{i \to \infty} (x_i)$ where $x_i \in X_i$. Given $e > 0$, $x_i \in X_e$ for all $i \geq e$ and therefore $y \in \overline{X_e}$. 
	
	Suppose now that $z \in \cap_{e = 0}^\infty \overline{X_e}$. This means that for all $e$ there is a sequence $(x_{e,i})_{i = 0}^\infty \sq X_e$ such that $z = \lim_{i \to \infty} (x_{e,i})$. Given $e$ we can therefore find an integer $i(e)$ such that $d(z, x_{e, i(e)}) \leq 1/e$. Then $z = \lim_{e \to \infty} (x_{e, i(e)})$ and since $x_{e, i(e)} \in X_e$ for all $e$, $z$ is in $Y$.
\end{proof}
\begin{lemma} \label{nu-to-alpha-lemma}
Let  $(\nu_e) \sq \N$ be a sequence of positive integers such that $\nu_{e + i} \equiv \nu_e \mod p^{e}$ for all $e$. Let $\alpha \in \Z_p$ be the $p$-adic limit of the sequence $(\nu_e)$. If $\alpha = \alpha_0 + p \alpha_1 + p^2 \alpha_2 + \cds $ is the $p$-adic expansion of $\alpha$ (i.e. $\alpha_j \in \{0, 1, \ds, p-1\}$ for all $j$), then for all $e$ we have
	$$\nu_e - \lfloor \frac{\nu_e}{p^e} \rfloor p^{e} = \alpha_0 + p \alpha_1 + \cds + p^{e - 1} \alpha_{e -1}.$$
\end{lemma}
\begin{proof}
	We have the following equivalences:
	$$\nu_e - \lfloor \frac{\nu_e}{p^{e}} \rfloor p^{e} \equiv \nu_e \equiv \alpha \equiv \alpha_0 + p \alpha_1 + \cds + p^{e -1} \alpha_{e- 1} \mod p^{e}.$$
	Since both $\nu_e - \lfloor \frac{\nu_e}{p^{e}} \rfloor p^{e}$ and $\alpha_0 + \cds + p^{e-1}\alpha_{e- 1}$ are integers between $0$ and $p^{e} - 1$ the statement follows.
\end{proof}
\subsection{Examples}
\begin{enumerate}[(1)]
	\item Let $p \geq 5$ and consider the principal ideal $\fa = (f)$ where $f = x^2 + y^3$. Then for all $e \geq 2$ the $\nu$-invariants are given by
	$$\nu^\bullet_f(p^e) = \begin{cases}
	\frac{5}{6}(p^e -1) + \N_0 p^e \text{ if } p \equiv 1 \mod 3 \\
	\frac{5}{6}p^e - \frac{1}{6}p^{e-1} - 1 + \N_0 p^e \text{ if } p \equiv 2 \mod 3
	\end{cases}$$
	(see \cite{MTW} and \cite[Ex. 3.4]{Mustata2009}). From Proposition \ref{prop-BSroots-newChar} we conclude that the Bernstein-Sato roots of $\fa$ are given by
	$$BS(\fa) = \begin{cases} \{-5/6, -1\} \text{ if } p \equiv 1 \mod 3 \\
	\{-1\} \text{ if } p \equiv 2 \mod 3. \end{cases}$$
	Observe the agreement with Bitoun's result: the Bernstein-Sato roots are precisely the $F$-jumping numbers in $(0, 1] \cap \Z_{(p)}$.
	
	\item Consider the monomial ideal $\fa = (x^2, y^3)$. Then we claim that the  $\nu$-invariants are given by
	$$\nu^\bullet_\fa(p^e) := \bigg\{ \lfloor \frac{ap^e - 1}{2} \rfloor + \lfloor \frac{bp^e - 1}{3} \rfloor :  a, b \in \N \bigg\}.$$
	Indeed, all the ideals $\cC^e_R \cdot \fa^s$ are monomial ideals and, therefore, $n$ is in $\nu^J_\fa(p^e)$ if and only if there exists some monomial $\mu = x^a y^b$ such that $\mu \in \cC^e_R \cdot \fa^n$ and $\mu \not \in \cC^e_R \cdot \fa^{n + 1}$. This is true, in turn, precisely when $n = \max\{s >0 : \mu^{p^e} \in \fa^s\}$. Since $\fa^s = (x^{2u} y^{3v}: u + v = s)$, this is equivalent to $n = \max\{u + v : u, v \in \N_0, 2u \leq ap^e -1, 3v \leq bp^e-1\}$, i.e. $n = \lfloor(ap^e - 1)/2 \rfloor + \lfloor (bp^e -1)/3 \rfloor$.
	
	\begin{enumerate}[(i)]
		\item Suppose that $p = 2$ and that $e$ is even. Then for all $a \in \N$ we have $\lfloor ap^{e} - 1)/2 \rfloor = a p^{e-1} - 1$, while
		$$\lfloor \frac{b p^e -1}{3} \rfloor = \begin{cases}
			c p^e - 1 \text{ if } b = 3c \\
			(c - \frac{1}{3})p^e - \frac{1}{3} \text{ if } b = 3c - 1 \\
			(c - \frac{2}{3})p^e - \frac{2}{3} \text{ if } b = 3 c - 2,
		\end{cases}$$
		where we always take $c \in \N$. We conclude that, for even $e$,
		\begin{align*}
		\nu^\bullet_\fa(p^e) & = \bigg\{a p^{e-1}+ c p^e - 2 : a, c \in \N \bigg\} \cup \bigg\{a p^{e-1}+ (c - \frac{1}{3}) p^e - \frac{4}{3}: a, c \in \N \bigg\} \\
			& \hspace*{50pt} \cup \bigg\{a p^{e-1}+ (c - \frac{2}{3}) p^e - \frac{5}{3}: a, c \in \N \bigg\} 
		\end{align*}
		and therefore $BS(\fa) = \{-4/3, -5/3, -2\}$ by Proposition \ref{prop-BSroots-newChar}.
		
		\item Suppose that $p = 3$. Then for all $e$ we have $\lfloor bp^e - 1 /3 \rfloor = bp^{e-1} - 1$, while
		$$\lfloor \frac{a p^e - 1}{2} \rfloor = \begin{cases}
		c p^e - 1 \text{ if } a = 2c \\
		(c - \frac{1}{2}) p^e - \frac{1}{2} \text{ if } a = 2c -1,
		\end{cases}$$
		where $c \in \N$, so we conclude that 
		$$\nu^\bullet_\fa(p^e) = \bigg\{ c p^e + bp^{e-1} - 2 \bigg\}  \cup \bigg\{ (c - \frac{1}{2})p^e + b p^{e-1} - \frac{3}{2} \bigg\}$$
		and therefore $BS(\fa) = \{-3/2, -2\}$. 
	\end{enumerate}
\end{enumerate}
\pagebreak
\subsection{Open questions}

The analogy with results from characteristic zero begs the following questions.

\begin{question}
	Suppose that the $F$-pure threshold $\alpha$ of $\fa$ lies in $\Z_{(p)}$. Is the largest Bernstein-Sato root of $\fa$ equal to $- \alpha$? 
\end{question}
We know the answer to this question is affirmative whenever $\fa$ is a principal ideal by \cite{Bitoun2018}, or when $\fa$ is a monomial ideal and $p \gg 0$ by \cite{QG19b}. 
\begin{question}
	In characteristic zero a root of the Bernstein-Sato polynomial comes with a multiplicity. Is there an analogue of the multiplicity of a root in characeristic $p$?
\end{question}

\begin{question}
	In \cite{Mustata2019}, \Mustata \ has shown that, over $\C$, if $\fa \sq R$ is an ideal generated by $(f_1, f_2, \ds, f_r)$ and we consider the element $F = f_1 y_1 + \cds + f_r y_r$ of $R[y_1, \ds, y_r]$ then $b_\fa(s) = b_F(s)/(s+1)$. In positive characteristic the Bernstein-Sato roots of $F$ must lie in $(0, 1]$ by \cite{Bitoun2018} so a completely analogous statement cannot be true, but we wonder if a similar statement can be found in positive characteristic.
\end{question}

\section{Appendix on $p$-adic expansions} \label{subsn-p-adic-exp}
We denote the $p$-adic integers by $\Z_p$. Recall that $\Z_p$ is the completion of $\Z_{(p)}$ at its maximal ideal $(p)$ and thus we have an inclusion $\Z_{(p)} \sq \Z_p$. We will say a $p$-adic number $\alpha$ is rational if $\alpha \in \Z_{(p)}$.

Given some $\alpha \in \Z_p$ we denote by $\alpha_i$ the unique integers with $0 \leq \alpha_i < p$ such that $\alpha = \alpha_0 + p \alpha_1 + p^2 \alpha_2 + \cds$. By an abuse of notation we will also denote by $\alpha_i$ the corresponding classes of these numbers in $\F_p$. We refer to the list $(\alpha_0, \alpha_1, \ds)$ as the expansion of $\alpha$. We say it is periodic (resp. eventually periodic) if there exists some $d$ such that $\alpha_{i + d} = \alpha_i$ for all $i \geq 0$ (resp. for all $i\geq 0$ large enough). We call such $d$ a period for $\alpha$.

Recall that given $\lambda \in \Q$ we have $\lambda \in \Z_{(p)}$ if and only if there exists some $d > 0$ with $\lambda(p^d - 1) \in \Z$ (for example one may apply Euler's theorem to the denominator of $\lambda$).

\begin{lemma} \label{periodic-iff-rational}
	Let $\alpha \in \Z_p$. Then the expansion of $\alpha$ is eventually periodic if and only if $\alpha$ is rational. Moreover, the list is periodic if and only if $\alpha$ is rational with $-1 \leq \alpha \leq 0$. A positive integer $d$ with $\alpha(p^d - 1) \in \Z$ is a period for $\alpha$.
\end{lemma}
\begin{proof}
	Let us first show that if $\alpha \in \Z_{(p)}$ with $-1 \leq \alpha \leq 0$ then this expansion is periodic. There exists some $d > 0$ and some $a \in \Z$ such that $\alpha = -a/(p^d - 1)$ with $0 \leq a < p^d$. Let $a = \sum_{i = 0}^{d-1} a_i p^i$ be the base $p$ expansion for $a$. Then
	\begin{align*}
	\alpha  = \frac{-a}{p^d - 1} & = (a_0 + \cds + p^{d-1} a_{d-1})(1 + p^d + p^{2d} + \cds) \\
	& = a_0 + \cds + p^{d-1}a_{d-1} + \\
	& + a_0p^d + \cds + p^{2d - 1} a_{d-1} + \cds
	\end{align*}
	and thus the expansion is periodic with period $d$. Conversely, if $\alpha$ has a periodic expansion then $\alpha = -a/(p^d -1 )$ for some $d$ and $a$ as above and thus $-1 \leq \alpha \leq 0$. 
	
	Now let $\alpha \in \Z_{(p)}$, and we will show it has an eventually periodic expansion. Observe it suffices to check the case where $\alpha_0 \neq 0$: if $\alpha$ has an eventually periodic expansion so does $p^n \alpha$. If $\alpha$ is negative we have $\alpha = \beta - N$ for some $-1 \leq \beta < 0$ and some $N \in \N_0$. By what we already proved, $\beta$ has a periodic expansion and, since it is negative, we have $\beta_i \neq 0$ for infinitely many $i$. It follows that for some $j \gg 0$ we have
	$$N \leq \beta_0 + p \beta_1 + \cds + \beta_{j-1} p^{j-1}$$
	and thus the expansions of $\alpha_i = \beta_i$ for all $i \geq j$. It follows that $\alpha$ has an eventually periodic expansion, as required. 
	
	Finally, the statement for negative numbers follows because whenever $\alpha_0 \neq 0$ the expansion for $- \alpha$ is given by
	$$- \alpha = (p - \alpha_0) + p(p - 1 - \alpha_1) + p^2(p-1-\alpha_2) + \cds$$ 
	It remains to show that if $\alpha$ has an eventually periodic expansion then $\alpha$ is rational. But again by adding the appropriate integer integer we may assume that $\alpha$ has a purely periodic expansion, and then the result follows from what we have already proved.
	
	The last statement is clear from the proof.
\end{proof}

\begin{lemma} \label{zp-to-padic-lemma}
	Let $\lambda \in \Z_{(p)}$ with $0 \leq \lambda \leq 1$ be such that $\lambda (p^d - 1) \in \N$ and let $\gamma := - \lambda$. Then for all $e > 0$ we have
	$$\gamma_0 + p \gamma_1  + \cds + p^{ed -1} \gamma_{ed -1} = \lambda(p^{ed} -1).$$
\end{lemma}
\begin{proof}
	Let $l := \lambda(p^d - 1)$. By assumption $l$ is an integer and $0 \leq l < p^d$. We can thus consider its base $p$-expansion
	$$l = l_0 + p l_1 + \cds + p^{d - 1} l_{d -1},$$
	where $0 \leq l_i < p$ for all $i$. In $\Z_p$ we have
	\begin{align*}
	\gamma = \frac{-l}{p^d -1 } & = l(1 + p^d + p^{2d} + \cds ) \\
	& = l_0 + pl_1 + p^2 l_2 + \cds + p^{d-1}l_{d-1} + \\
	& + p^d l_0 + p^{d+1} l_1 + p^{d+2} l_2 + \cds + p^{2d -1 }l_{d-1} + \\
	& + p^{2d} l_0 +  \cds 
	\end{align*}
	and thus this is the $p$-adic expansion for $\gamma$. It follows that
	$$\gamma_0 + p \gamma_1 \cds + p^{ed -1} \gamma_{ed -1} = l(1 + p + \cds + p^{ed -1}) = \frac{l(p^{ed}-1)}{(p^d - 1)} = \lambda(p^{ed} - 1)$$
	as required.
\end{proof}

\begin{lemma} \label{p-adic-lemma3}
	Let $\alpha \in \Z_{(p)}$. Then there is some $\gamma \in \Z_{(p)}$ with $-1 \leq \gamma \leq 0$ such that the $p$-adic expansions of $\alpha$ and $\gamma$ are eventually equal. In particular, there is some $n \in \Z$ such that $\alpha = n + \gamma$.
\end{lemma}
\begin{proof}
	By Lemma \ref{periodic-iff-rational} we know that the expansion for $\alpha$ is eventually periodic. This means we may write $\alpha = m + p^n \gamma$ where $0 \leq m < p^n$ and $\gamma \in \Z_p$ has a periodic expansion with period $d >0$ -- in particular, $-1 \leq \gamma \leq 0$ by Lemma \ref{periodic-iff-rational}. By making $m$ bigger if necessary we may furthermore assume that $d$ divides $n$. With these assumptions $\alpha$ and $\gamma$ have expansions that are eventually equal and thus the statement follows.
\end{proof}

\begin{remark}
	If $\alpha$ is not an integer then $\gamma$ is the unique representative of $\alpha + \Z$ contained in $(-1,0)$. If $\alpha$ is a nonnegative integer then $\gamma = 0$ and if $\alpha$ is a negative integer then $\gamma = -1$.
\end{remark}

\bibliography{biblio.bib}
\bibliographystyle{alpha}

\end{document}